\documentclass[11pt]{article}
\usepackage[utf8]{inputenc}
\usepackage{amsthm}
\usepackage{amsfonts}
\usepackage{amsmath}
\usepackage{mathtools}
\usepackage{array}
\usepackage{verbatim}
\usepackage{cases}
\usepackage{url}
\usepackage{tagging, comment, verbatim, booktabs}
\usepackage[normalem]{ulem}
\usepackage{cancel}
\usepackage{accents}
\usepackage{fancyhdr}
\usepackage{datetime}
\usepackage{mathrsfs}
\usepackage{xspace}
\usepackage{pstricks}

\usepackage{chngcntr}
\usepackage{apptools}
\AtAppendix{\counterwithin{lemma}{section}}

\usepackage[inline]{enumitem}
\usepackage{tikz-cd} 
\usepackage{graphicx}

\mathtoolsset{showonlyrefs=true} 

\numberwithin{equation}{section} 

\makeatletter
\newcommand{\pushright}[1]{\ifmeasuring@#1\else\omit\hfill$\displaystyle#1$\fi\ignorespaces}
\newcommand{\pushleft}[1]{\ifmeasuring@#1\else\omit$\displaystyle#1$\hfill\fi\ignorespaces}
\makeatother

 \makeatletter
 \DeclareRobustCommand{\cev}[1]{%
   \mathpalette\do@cev{#1}%
 }
 \newcommand{\do@cev}[2]{%
   \fix@cev{#1}{+}%
   \reflectbox{$\m@th#1\vec{\reflectbox{$\fix@cev{#1}{-}\m@th#1#2\fix@cev{#1}{+}$}}$}%
   \fix@cev{#1}{-}%
 }
 \newcommand{\fix@cev}[2]{%
   \ifx#1\displaystyle
     \mkern#23mu
   \else
     \ifx#1\textstyle
       \mkern#23mu
     \else
       \ifx#1\scriptstyle
         \mkern#22mu
       \else
         \mkern#22mu
       \fi
     \fi
   \fi
 }

\makeatother

\newcounter{alphasect}
\def\alphainsection{0}

\let\oldsection=\section
\def\section{%
  \ifnum\alphainsection=1%
    \addtocounter{alphasect}{1}
  \fi%
\oldsection}%

\renewcommand\thesection{%
  \ifnum\alphainsection=1%
    \Alph{alphasect}%
  \else%
    \arabic{section}%
  \fi%
}%

\newenvironment{alphasection}{%
  \ifnum\alphainsection=1%
    \errhelp={Let other blocks end at the beginning of the next block.}
    \errmessage{Nested Alpha section not allowed}
  \fi%
  \setcounter{alphasect}{0}
  \def\alphainsection{1}
}{%
  \setcounter{alphasect}{0}
  \def\alphainsection{0}
}%

\theoremstyle{definition}
\newtheorem{theorem}{Theorem}[section]
\newtheorem{corollary}[theorem]{Corollary}
\newtheorem{lemma}[theorem]{Lemma}
\newtheorem{prop}[theorem]{Proposition}
\newtheorem{conjecture}[theorem]{Conjecture}
\newtheorem{rem}[theorem]{Remark}
\newtheorem{definition}[theorem]{Definition}
\newtheorem{obs}[theorem]{Observation}

\begin{document}
\title{A rigorous formulation of and partial results on Lorenz's ``consensus strikes back'' phenomenon for the Hegselmann-Krause model  
}
\author{Edvin Wedin}


\usetag{proofsprel}
\usetag{proofslemmas}
\usetag{proofsuniform}

\maketitle

\abstract{In a 2006 paper, Jan Lorenz observed a curious behaviour in numerical simulations of the Hegselmann-Krause model: Under some circumstances, making agents more closed-minded can produce a consensus from a dense configuration of opinions which otherwise leads to fragmentation. Suppose one considers initial opinions equally spaced on an interval of length $L$. As first observed by Lorenz, simulations suggest that there are three intervals $[0, L_1)$, $(L_1, L_2)$ and $(L_2, L_3)$, with $L_1 \approx 5.23$, $L_2 \approx 5.67$ and $L_3 \approx 6.84$ such that, when the number of agents is sufficiently large, consensus occurs in the first and third intervals, whereas for the second interval the system fragments into three clusters. In this paper, we prove consensus for $L \leq 5.2$ and for $L$ sufficiently close to 6. These proofs include large computations and in principle the set of $L$ for which consensus can be proven using our approach may be extended with the use of more computing power.
 We also prove that the set of $L$ for which consensus occurs is open.
Moreover, we prove that, when consensus is assured for the equally spaced systems, this in turn implies asymptotic almost sure consensus for the same values of $L$ when initial opinions are drawn independently and uniformly at random. We thus conjecture a pair of phase transitions, making precise the formulation of Lorenz's ``consensus strikes back'' hypothesis. Our approach makes use of the continuous agent model introduced by Blondel, Hendrickx and Tsitsiklis. Indeed, one contribution of the paper is to provide a presentation of the relationships between the three different models with equally spaced, uniformly random and continuous agents, respectively, which is more rigorous than what can be found in the existing literature.
}

\section{Introduction}
In the classical Hegselmann-Krause model (the HK-model for short) in opinion dynamics, each agent $i$ in a set of agents indexed by integers $1, 2, 3, ..., n$ possesses an opinion $f_t(i)$ at time $t$. All agents then simultaneously update their opinion at the next time step according to the rule

\begin{equation}\label{eq:intro1}
f_{t+1}(i) = \frac{1}{ |\mathcal{N}_i|} \sum_{j \in \mathcal{N}_i} f_t(j)
\end{equation}
where $\mathcal{N}_i = \{j \in [n] : | f_t(j) - f_t(i)| \leq 1\}$, and $[n] = \{1,\, 2,\, ...,\, n\}$.

The paper normally cited in connection to this model is \cite{HK}, which presents simulations and some important basic results.
Strictly speaking, \cite{HK} gives a slightly different definition where the 1 in the expression for $\mathcal{N}_i$ is replaced by a \emph{confidence radius} $r$. We note that simultaneously scaling $r$ along with all opinions does not change the qualitative behaviour of the model, and the formulation given here, referred to as the \emph{normalised} model, is common.
 When discussing the HK-model, it is useful to employ the concept of the \emph{connectivity graph}, which takes as nodes the agents and connects the agents $i$ and $j$ precisely when $j \in \mathcal{N}_i$.

Perhaps the most basic observation is that if two agents hold opinions separated by more than 1, and no other agents holds an opinion in between, the two will never interact.
 A second, slightly less obvious, observation is that even if the current state has a connected connectivity graph, that of the updated state might be disconnected, as may be readily verified by assigning the opinions 0, 0, 1, 2, 3 and 3 to six agents and computing the update. Breaking of the connectivity graph, which by the first observation is irreversible, is referred to as \emph{fragmentation}. A third observation of \cite{HK}, which requires a little more mathematical work to verify, is that for each possible initial choice of opinions, there is some finite number $T$ such that after updating the system $T$ times, the opinion profile reaches a fixed point which is not changed by subsequent updates. 
 When we reach such a fixed point, we say that the system \emph{freezes}\footnote{The study of the time needed for freezing in the Hegselmann-Krause model has spawned at least half a dozen papers by about as many authors. State of the art results can be found in \cite{BBCN}, \cite{MT} and \cite{WH1}.},
 and it is not hard to verify that
 a frozen state must consist of a set of \emph{clusters}, where agents in a cluster are in agreement and clusters are pairwise separated by strictly more than 1.
 A configuration consisting of a single cluster is called a \emph{consensus}.

The HK-model has received considerable attention, and the original paper \cite{HK} has close to 3000 citations on Google Scholar at the time of writing. Most of the citing papers present simulations of all sorts of variations on the original model.
There are, to date, only a handful of papers with rigorous mathematical results for the basic model, e.g.: \cite{BBCN}, \cite{MT}, \cite{WH1}, \cite{WH2}, \cite{BHT2}.

In many instances, interesting hypotheses have first arisen from simulations. One particularly nice example of this concerns the question of what the final configuration typically looks like when the model is initiated with a large number $n$ of agents equidistributed on the interval 
$[0,\,L]$, for some fixed $L$. 
In a seminal 2006 paper \cite{L}, Lorenz approached this problem in two ways, the first of which was to simply simulate the dynamics of (1) for equally spaced agents on various intervals, including half-infinite ones. The second way was to devise a clever interactive Markov chain (IMC) model where the opinion space is discretised and the agents change sections according to a stochastic matrix chosen so as to mimic the original behaviour of the model, arguing that the models should, intuitively, be equal in the limit when the discretisation is refined. In this way, he produced an early way of simulating not the dynamics of the actual agents, but rather that of their \emph{distribution}. The point is that this, at least morally, should hint at the typical behaviour of the actual model for large numbers of agents.

The paper contains no mathematical proofs, but various interesting observations and remarks on the presented simulations.

One of Lorenz's observations, which gave the paper its title, is that, in his IMC model, the resulting configuration of clusters behaves unexpectedly when the radius of confidence is varied. Adhering to the convention of using a normalised radius, which we will keep throughout this paper, his finding translates to the following. When opinions are spread on an interval of length $L<1$, all agents reach a consensus, and this remains true for a while when $L$ grows larger than 1. At around $L \approx 5$, the final configuration undergoes a bifurcation, and changes from one to three clusters. What is even more interesting is that around $L \approx 6$ the system undergoes another bifurcation, and the final state returns to consisting of a single cluster. In the words of Lorenz, consensus ``strikes back''!

The following conjecture is implicit in Lorenz's paper:
\begin{conjecture}\label{con:Lorenz}
  Denote by $C_\text{Uni}(L,n)$ the random final number of clusters reached by updating according to \eqref{eq:intro1} when starting from $n$ agents whose opinions are drawn uniformly and independently at random from the interval $[0,\, L]$.

  Then the limit $C_\text{Uni} (L) = \lim_{n \to \infty} C_\text{Uni}(L,\, n)$ exists as a random variable and there exist numbers $0 < L_1 < L_2 < L_3$, such that:
  \begin{enumerate}
  \item If $0 \leq L < L_1$, then $C_\text{Uni}(L) = 1$ almost surely.
  \item If $L_1 < L < L_2$,  then $C_\text{Uni}(L) = 3$ almost surely.
  \item If $L_2 < L < L_3$,  then $C_\text{Uni}(L) = 1$ almost surely.
  \end{enumerate}
\end{conjecture}

Lorenz discusses his observation in relation to a 2004 conjecture by Hegselmann, stating that for any $L$ there might be a number $n$ such that $n$ equally spaced agents on an interval of length $L$ must eventually reach a consensus, a conjecture that is still not disproven rigorously.

A step forward in our understanding of equidistributed agents on an interval of length $L$ was taken in a 2007 paper of Blondel, Hendrickx and Tsitsiklis \cite{BHT2}. The authors choose another approach for studying the dynamics of the distribution of agents, namely to consider a continuum of agents and index them not by a set of natural numbers, but by an interval $I$ of real numbers. Replacing the sum in \eqref{eq:intro1} by an integral, the analogue of \eqref{eq:intro1} is then that for every agent $\alpha \in I$, its updated opinion is given by

\begin{equation}\label{eq:update}
f_{t+1}(\alpha) = \frac{1}{\mu (\mathcal{N}_\alpha)}  \int_\mathcal{N_\alpha} f_t(\beta) d\beta
\end{equation}
where $\mathcal{N}_\alpha = \{ \beta \in I : |f_t(\alpha) - f_t(\beta)| \leq 1\}$.

In contrast to the IMC model of Lorenz this formulation doesn't require a finite discretisation of the opinion space. The downside is that it is hard to use for actual formal computations, but it is very useful from a theoretical point of view.

~

The three chief contributions in this paper are
\begin{enumerate}
\item to develop techniques for finding rigorous bounds on how much the evolution of a finite number $n$ of equally spaced agents on an interval of length $L$ may differ from the limiting case when $n$ goes to infinity,
\item to give a rigorous presentation of the relationships between the three different models with equally spaced, uniformly random and continuous agents, respectively,
\item using (i) and (ii), to prove the following theorem and corollary:
\end{enumerate}

\begin{theorem}\label{thm:intro1}
  Denote by $C_\text{Eq}(L,n)$ the final number of clusters reached by updating according to \eqref{eq:intro1} when starting from $n$ agents whose opinions are equally spaced on the interval $[0,\, L]$, and let
  \begin{align*}
      C_\text{Eq} (L) = \lim_{n \to \infty} C_\text{Eq}(L,\, n),
  \end{align*}
whenever the limit exists, i.e. when $C_\text{Eq}(L,\,n)$ is constant for all sufficiently large $n$.
  
\begin{enumerate}
  \item $\{L>0 : C_{Eq}(L) = 1\}$ is an open set.
  \item If $L_1' = 5.2$, $L_2' = L_3' = 6$, there exists some number $T$ such that, if $L \in [0,\, L_1'] \cup [L_2',\, L_3']$ then, for all sufficiently large $n$, the $T$'th update of the corresponding equally spaced profile is a consensus.

    Hence, there exist numbers $ 5.2 < L_1 < L_2 < 6 < L_3$ such that if $L \in [0,\, L_1) \cup (L_2,\, L_3)$ then $C_\text{Eq}(L) = 1$.
\end{enumerate}
\end{theorem}

\begin{corollary}\label{cor:intro2}
  Denote by $C_\text{Uni}(L,n)$ the random final number of clusters reached by updating according to \eqref{eq:intro1} when starting from $n$ agents whose opinions are drawn uniformly and independently at random from the interval $[0,\, L]$, and let
  \begin{align*}
    C_\text{Uni} (L) = \lim_{n \to \infty} C_\text{Uni}(L,\, n),
  \end{align*}
  whenever the limit random variable exists.
  Then
  \begin{enumerate}
    \item $\{L>0 : C_{Uni}(L) = 1\}$ is an open set.
  \item With the same numbers $L_i$, $L_i'$ and $T$ as in Theorem \ref{thm:intro1} we have that, if $L \in [0,\, L_1) \cup (L_2,\, L_3)$, then $C_{\text{Uni}}(L) = 1$ and, if $L \in [0,\, L_1'] \cup \{ 6 \}$ then, as $n\to \infty$, the  $(T+2)$'th update of the corresponding uniformly random profile is a consensus asymptotically almost surely.
  \end{enumerate}
\end{corollary}

We will build on the results of \cite{BHT2} in several ways, and refer to that paper for some proofs and additional background.

The rest of the paper will be structured as follows:

{\bf Section 2} will serve as a theoretical foundation. Here, we will develop a rigorous theory of opinion profiles, both in the traditional discrete case, i.e. when the number of agents is finite, and in that of an agent continuum, as well as tools to relate the two. In particular, we will introduce the concepts of \emph{refining} and \emph{coarsening}, which will be used heavily to handle and relate different deterministic samples from the same distribution. This section also presents, and in some cases strengthens, some previously known results that will be used.
An important result (Proposition \ref{lem:cont}) is that the updating operation \eqref{eq:update} is continuous, with respect to the infinity norm, at so-called \emph{regular} profiles (Definition \ref{def:regularity}).
At the end of the section, we prove part (i) of Theorem \ref{thm:intro1} and show how Corollary \ref{cor:intro2} follows from Theorem \ref{thm:intro1}.

We are then left to prove part (ii) of Theorem \ref{thm:intro1} in subsequent sections. Our basic strategy is to reduce the proof to a finite computation\footnote{What we mean by this is a computation that is certainly finite and, if it produces a certain result, allows us to deduce the theorem.}.
To do so, we need to go beyond the general theory of Section 2 and develop explicit quantitative bounds when comparing the updates of a discrete profile and small perturbations of it. In particular, we compare updates of a discrete profile and its refinements. This material is presented in {\bf Section 3}.

In {\bf Sections 4 and 5}, we apply the results of Section 3 to the case of equally spaced opinions. In order to ensure that the resulting finite computations are manageable, we will also use a result from Section 2 (Corollary \ref{cor:regularity}).

In {\bf Section 4} we prove the existence of $L_1'$. This involves a large but manageable number of computations for a grid of $L$-values up to $L = 5.2$. In principle we could push beyond 5.2, but as one gets closer to the conjectured phase transition at about 5.23 the amount of computing power needed increases drastically.

Consensus at $L = 6$ is proven in {\bf Section 5}. This time, to reduce the proof to a \emph{manageable} computation requires more than just applying the theory from Section 3.
 Lorenz already observed that the mechanism by which consensus is reached after it strikes back is different than for smaller values of $L$. The profile quickly settles into a state where most agents reside in 5 groups and from there the span of opinions shrinks \emph{very slowly}. The error analysis from Section 3 is no longer practical over such time scales. Hence, we prove a theorem (Theorem \ref{thm:6tocons}) saying, informally, that a certain class of profiles in which most agents reside in 5 groups must evolve to consensus. This effectively means that we just need to compute the updates of a single profile until the conditions in the theorem are satisfied, and then use the error analysis from Section 3. This turns out to lead to a manageable computation.

 All our computations are carried out in the high precision ball arithmetic of the package Arb\cite{Arb}. All code is written in Julia\cite{Julia}.
 
{\bf Section 6} contains a discussion of our results and of possible future work.

\section{Definitions and  results}
In what follows, we will find it convenient to adopt a notation that differs slightly from \eqref{eq:intro1}.
Still, we consider a set $[n] = \{1,\, 2,\, ...,\, n\}$ of agents, and their ``opinions'' $f(1),\, ...,\, f(n)$. For the updates we will follow the notation of \cite{BHT2} which uses the \emph{updating operator} $U$, defined as
\begin{equation}\label{eq:update0}
  Uf(i) = \frac{1}{\# \left(\mathcal{N}_i(f)\right)}\sum_{j \in \mathcal{N}_i(f)} f(j)
\end{equation}
where $\mathcal{N}_i(f) = \{j: |f(j)-f(i)| \leq 1\}$.

This formulation, clearly equivalent with that given by \eqref{eq:intro1}, will be referred to as \emph{the traditional model}, and we will frequently describe results and procedures in terms of this, although the setting formally will be more general.

We begin by introducing some notation: Let $\langle A(i) \rangle_{i \in B}$ denote the average of a function $A$ with values taken from a non-empty set $B$. With this notation, \eqref{eq:update0} is condensed to
\begin{equation}\label{eq:update0b}
  Uf(i) = \langle f(j) \rangle_{j \in \mathcal{N}_i(f)}.
\end{equation}

In what follows, we will use a more general formalism, largely following \cite{BHT2}. This is to be better able to compare the behaviour of the model for different values of $n$, and to relate this to the resulting behaviour if we let $n \to \infty$. 

\begin{definition}\label{def:sim}
  If, for two bounded and Lebesgue measurable functions $f,\, g :[0,\,1] \to \mathbb{R}$, there exists a measure preserving bijection $\sigma:[0,\, 1] \to [0,\, 1]$ such that $f = g \circ \sigma$, we say that $f$ and $g$ are \emph{permutation equivalent} and write $f \sim g$.
\end{definition}

\begin{obs}
  If two functions $f$ and $g$ are permutation equivalent, it follows that $\mu \left( f^{-1} \left( \left( -\infty,\, x\right] \right) \right) = \mu \left( g^{-1} \left( \left( -\infty,\, x \right] \right) \right)$ for all $x \in \mathbb{R}$.
\end{obs}

It is easy to check that $\sim$ is an equivalence relation.

\begin{definition}
  An \emph{(opinion) profile} $f$ is a non-decreasing function $[0,\, 1] \rightarrow \mathbb{R}$.

  An element $\alpha$ of the unit interval will be referred to as an \emph{agent}, and $f(\alpha)$ will be referred to as the \emph{opinion} of the agent $\alpha$.

  The set of opinion profiles is denoted by $\mathscr{O}$.
\end{definition}

These profiles will be updated according to the following adaptation of the rule \eqref{eq:update0b}:

\begin{definition}\label{def:update1}
  The \emph{updating operator} $U: \mathscr{O} \to \mathscr{O}$ takes a profile $f$ to its update $Uf$ according to the rule 
  \begin{equation*}
    Uf(\alpha) =
    \begin{cases}
      \langle f(\beta) \rangle_{\beta \in \mathcal{N}_\alpha(f)}, & \text{if }\mu(\mathcal{N}_\alpha(f)) > 0 \\
      f(\alpha), & \text{otherwise,}
    \end{cases}
  \end{equation*}
  where $\mathcal{N}_\alpha(f) = \{\beta: |f(\beta)-f(\alpha)| \leq 1\}$. The average $\langle \cdot \rangle$ over an interval $A$ of positive Lebesgue measure $| A |$ is given by
  \begin{equation}\label{eq:update1}
     \langle f(\beta) \rangle_{\beta \in A} = \frac{1}{|A|} \int_A f(\beta) d\beta.
  \end{equation}
\end{definition}

\begin{obs}\label{obs:deriverbar}
  It is easy to check that if $f$ is non-decreasing so is $Uf$, so for any profile $f$ and any natural number $t$ the $t$-fold update $U^t f$ is well defined.
\end{obs}
\begin{obs}\label{obs:transinv}
From the definition, we see immediately that $U$ is \emph{translation invariant}, in the sense that
  \begin{align*}
  U(f + C) = Uf + C
  \end{align*}
 for any profile $f$ and any $C \in \mathbb{R}$.
\end{obs}

\begin{obs}\label{obs:permutationupdate}
Though the operator $U$ will mainly be used for profiles, the same definition can be made for all measurable functions. With this in mind, we will occasionally without comment let $U$ act on a measurable function without checking whether or not it's non-decreasing. 
  In particular, note that for any profile $g$ and any measurable function $h$, we have $g \sim h \implies Ug \sim Uh$.
\end{obs}

For an agent $\alpha$ and a profile $f$, we will refer to the set $\mathcal{N}_\alpha(f)$ as the \emph{neighbourhood of $\alpha$}, and to the members of said set as the \emph{neighbours of $\alpha$}. For two agents $\alpha$ and $\beta$, we will also say that $\alpha$ can \emph{see} $\beta$, or that $\beta$ is \emph{within sight} of $\alpha$, if and only if $\beta \in \mathcal{N}_\alpha(f)$.

The following definition lets us use this formalism to emulate the traditional model:

\begin{definition}
  A \emph{discrete pre-profile on $n$ agents} is a function $ [0,\,1] \rightarrow \mathbb{R}$ which, for every integer $2 \leq i \leq n$,  is constant on the interval $(\frac{i-1}{n},\, \frac{i}{n}]$, as well as on the interval $[0,\, \frac{1}{n}]$.

    A \emph{discrete profile on $n$ agents} is a profile which is also a discrete pre-profile on $n$ agents.

  For a discrete profile $g$ on $n$ agents, we will let the term \emph{agent} refer to an interval of the form $(\frac{i-1}{n},\, \frac{i}{n}]$, for $2 \leq i \leq n$, or $[0,\, \frac{1}{n}]$
. 
          
  The set of discrete profiles on $n$ agents is denoted by $\mathscr{O}_n$.
\end{definition}
For discrete profiles, we will abuse notation by referring to agents by their index, and often adopt the shorthand notation of writing $g(i)$ instead of $g(\frac{i}{n})$ and $\mathcal{N}_i$ instead of $\mathcal{N}_{\frac{i}{n}}$ when there is no risk of confusion. 

It should be clear from the context which of the two notations is being used, but 
as a rule we will use the Latin indices $i$, $j$, $n$ and so on to denote integers, whenever the shorthand is used, and Greek letters or fractions otherwise.

As observed in Section 1, a well known property of the traditional Hegselmann-Krause model is that any profile with a finite number of agents must \emph{freeze}, that is reach a fixed point, in finite time. This can be summarised as follows. 

\begin{obs}\label{obs:freezing}
  Let $f$ be a discrete profile. Then there exists $T \in \mathbb{N} \cup \{0\}$ such that $U^T f = U^{T+t} f$ for any $t \in \mathbb{N}$. The smallest such $T$ is called the \emph{freezing time} of $f$.
\end{obs}
\begin{definition}
   For a profile $f$, let $U^\infty f$ denote the pointwise limit $\lim_{t \to \infty } U^t f$, whenever the limit exists.
\end{definition}
By Observation \ref{obs:freezing}, $U^\infty f$ is well defined for any discrete profile $f$. It would follow from Conjecture 2 in \cite{BHT2} that it is well defined for any profile $f$, but this fundamental problem remains unsolved.

On the way to freezing, the agents in a profile will typically agregate into \emph{clusters}, and we make the following definition.
\begin{definition}
   In a profile $f$, a maximal set of agents which share the same opinion, that is, a maximal set of agents $\{\alpha : f(\alpha) = x\}$ for some $x$, is called a \emph{cluster}.

   A profile where all agents lie in a single cluster, i.e. a constant profile, is referred to as a \emph{consensus}, and, given a time $t$, a profile $f$ such that $U^t f$ consists of a single cluster is said to have \emph{reached a consensus} at time $t$.
\end{definition}

To have a formal way of manipulating profiles, we make the following definition.
\begin{definition}
  Given a discrete profile $f$ on $n$ agents, \emph{moving} an agent $i \in \{1,\, 2,\, ...,\, n\}$ will refer to the act of changing the value of $f(\alpha)$ for all $\alpha$ in the interval corresponding to the index $i$ to some common value, and replacing the resulting function $f_2$ with a profile $f_3$ such that  we have the equivalence $f_2 \sim f_3$ with the relation from Definition \ref{def:sim}. The amount by which $f(\alpha)$ is changed will be referred to as the amount \emph{by} which $\alpha$ was moved.
\end{definition}
We note that, in this definition, we will have $f_2 =f_3$ if the initial change in value preserves the non-decreasing quality of $f$. 

The following definitions present the new notions of coarsening and regular refinement of profiles, which will be central in the proofs to come.

\begin{definition}
  Given a discrete profile $f$ on $n$ agents and $k \in \mathbb{N}$, a \emph{$k$-regular refinement} of $f$ is a profile $g$ on $n+(n-1)k$ agents, such that $f(\frac{i}{n}) = g(\frac{i+(i-1)k}{n+(n-1)k})$ for any $i = 1,\,2,\, ... ,\,n$.

  If $g$ is a  $k$-regular refinement of $f$ for some $k$ it will sometimes be referred to as just a \emph{regular refinement}, without specifying for which $k$.
\end{definition}

\begin{definition}
  Given a discrete profile $f$ on $n$ agents, the \emph{canonical $k$-regular refinement} $f^{(k)}$ is the $k$-regular refinement of $f$ for which the sequence $\left( f^{(k)}(\frac{i + (i-1)k + j}{n + (n-1)k}) \right)_{j=0}^{k+1}$ constitutes an arithmetic progression for each $i = 1,\,2,\, ...,\, n-1$.
\end{definition}

In terms of the traditional model,
regularly refining a profile means adding some fixed number $k$ of new agents between every existing pair of consecutive agents.

The canonical $k$-regular refinements are those which, in some sense, are the most spread out. 
They represent a linear interpolation of the opinions in a discrete profile.

\begin{definition}\label{def:coarsening}
    For any function $g: [0,\,1] \to \mathbb{R}$, $n \geq 1$ and $k \geq 0$, we define the \emph{$k$-coarsening} of $g$, $B_k^n(g)$, as the discrete pre-profile on $n$ agents which satisfies $B_k^n(g)(\frac{i}{n}) =  g(\frac{i+(i-1)k}{n+(n-1)k})$ for each $i \in [n]$. When the $k$ is not specified, or otherwise where there is no risk of confusion, we will simply refer to \emph{coarsenings.}

    Further, we define the \emph{limit coarsening} of $g$, $B_\infty^n(g)$, as the discrete pre-profile on $n$ agents which satisfies $B_\infty^n(g)(\frac{i}{n}) = g(\frac{i-1}{n-1})$ for each $i \in [n]$.
\end{definition}

The way to think about Definition \ref{def:coarsening} is that, given any discrete profile $f$ on $n$ agents, the operator $B_k^n$ takes any $k$-regular refinement of $f$ and returns $f$. For instance, for any discrete profile $f$ on $n$ agents and any $k \geq 0$ we have that $B_k^n(f^{(k)}) = f$. Also note that a coarsening of a profile is always a profile.

As for the limit coarsening, it should be thought of mainly as an operator to use on \emph{regular} profiles, defined below. It can for instance be instructive to note that if we define $f^{(\infty)} = lim_{k \to \infty} f^{(k)}$ as the pointwise limit of the $k$-regular refinements of $f$, then $B_\infty^n(f^{(\infty)}) = f$.

For a reader familiar with signal processing, yet another way to view the concepts of refining and coarsening is to consider profiles as signals. The two then roughly correspond to (admittedly degenerate) upsampling and downsampling, respectively.

\begin{definition}
    For a positive real number $L$, we define the \emph{canonical linear profile of diameter $L$} by $f^L (\alpha) = L\alpha$.

    For $n \in \mathbb{N}$, the limit coarsening $B_\infty^n(f^L) := f^{n,L}$ will be called the \emph{canonical equally spaced profile on $n$ agents with diameter $L$}. Thus, $f^{n,L}$ consists of $n$ agents equally spaced on the interval $[0,\,L]$, i.e.: $f^{n,L}(i) = \frac{(i-1)L}{n}$ for $i = 1,\, ...,\, n$.
    
 We note that, for any $k \geq 0$, $(f^{n,L})^{(k)} = f^{n+(n-1)k,L}$.
\end{definition}

The following three definitions will be much employed throughout the whole paper. The third is an essential cornerstone in the theory we need to prove Theorem \ref{thm:intro1}.

\begin{definition}
  Given $c \in \mathbb{R}$, a profile $f$ is said to be \emph{symmetric about c} if $f(\alpha) + f(1 - \alpha) = c$ for almost every\footnote{That is, outside of a set of measure zero.} 
 $\alpha \in [0,\, 1]$.
\end{definition}
We do not require $f(\alpha) + f(1 - \alpha) = c$ \emph{everywhere}, as this would not allow us to speak of symmetric discrete profiles.
\begin{obs}\label{obs:symmetry}
If $f$ is symmetric about $c$, then so is $U^t f$ for any $t \geq 0$.
\end{obs}

\begin{definition}
  The \emph{diameter} of a profile $f$ is defined as $D(f) = f(1) - f(0)$.
\end{definition}

\begin{definition}\label{def:regularity}
  Given a set $S \subseteq [0,\, 1]$, an injective function ${f:  S \to \mathbb{R}}$ is said to be \emph{regular on $S$}  if there exist strictly positive real numbers $m$ and $M$, such that
  \begin{equation*}
    m \leq \frac{|f(\alpha) - f(\beta)|}{|\alpha - \beta|} \leq M
  \end{equation*}
  for any distinct $\alpha,\, \beta \in S$.

  A function that is regular on the whole set $[0,\, 1]$ is simply called regular, and the term $(m,\,M)$-regular is used when there are specific numbers $m$ and $M$ which satisfy the above inequalities.
\end{definition}

We stress that these parameters $m$ and $M$ are not defined in the same way as in \cite{BHT2}. Our $m$ and $M$ correspond to $\frac{1}{M}$ and $\frac{1}{m}$ in \cite{BHT2}.

In this paper all regular functions will be non-decreasing. In this case, another way of phrasing the definition is that $m$ and $M$ act as lower and upper bounds, respectively, on the derivative $f'$ of $f$ wherever it is defined. Yet another way is to say that $M$ and $\frac{1}{m}$ are Lipschitz constants for $f$ and $f^{-1}$, respectively.

Evidently, a discrete profile cannot be regular, but any regular profile must instead be continuous. 

\begin{rem}
A reader might ask: ``In order to prove Theorem \ref{thm:intro1}, why don't you just compute explicit formulas for $U^t f^{n,L}$, alternatively $U^t f^L$, for general $t$, $n$ and $L$?'' The short answer is that, though it might be possible in principle, the calculations quickly become messy as $t$ increases.

Consider an arbitrary regular profile $f$ and an agent $\alpha$. If $f$ is differentiable at the three agents $\alpha$, $f^{-1}(\alpha - 1)$ and $f^{-1}(\alpha + 1)$, then $U f(\alpha)$ is also differentiable. If $f$ has a corner at exactly one of the three agents, however, $U f$ will have a corner at $\alpha$. Heuristically, each corner should have three opportunities, or two if its opinion is close to that of an extremist, to beget another corner. Counting the endpoints as corners, we conclude that the number of corners of $U^t f^L$ should lie between $2^t$ and $3^t$.

To even further complicate the matter, the expression on each piece quickly grows unmanageable as well, and already after a few updates it's nontrivial to write them in terms of elementary functions.

Similar remarks apply to the discrete profiles $U^t f^{n,L}$. In Appendix A we present formulas for $t = 1,\, 2$. By studying these formulas, we think it is clear that this is not a fruitful strategy for general $t$.
\end{rem}

A central topic in this text is that of \emph{random profiles}, by which we mean profile-valued random variables. These will be generated by drawing a number $n$ of opinions independently at random from some probability distribution, sorting them, and creating a profile with $n$ agents holding the drawn opinions. In working with these random profiles, we will use the following, very helpful, lemma:
\begin{lemma}[Glivenko--Cantelli (see for instance \cite{vdV1}, p 266)]\label{lem:glivcan}
  Let $F$ be the cumulative distribution function of some real valued random variable and let $F_n$ be the empirical distribution function for a sample of size $n$. Then
  \begin{equation}
    ||F_n - F||_\infty 
    \to 0 
  \end{equation}
  asymptotically almost surely (a.a.s.), i.e. almost surely when $n \to \infty$. 
\end{lemma}

Note that, if the random variable in question is bounded, the quantile function given by
\begin{align*}
Q_F(\alpha) = &
\begin{cases}
\inf \{x: F(x) \geq \alpha\} & \text{if } \alpha > 0\\
\lim_{\alpha \to 0} Q_F(\alpha) & \text{if } \alpha = 0
\end{cases}
\end{align*}
is a profile. Further, the empirical quantile function for a sample of size $n$, given by
\begin{equation}\label{eq:quantile}
  f_n (\alpha) = Q_{F_n}(\alpha)
\end{equation}
  is a discrete profile with $n$ agents.

The following is immediate.

\begin{corollary}\label{cor:glivcan}
With $F$ as in Lemma \ref{lem:glivcan}, if $F$ is regular, then
\begin{equation}
||Q_F - f_n||_\infty \to 0
\end{equation}
a.a.s. as $n \to \infty$.
\end{corollary}
\begin{proof}
See for instance \cite{vdV1} p. 305.
\end{proof}

For any profile $f$ the set $\mathcal{N}_0(f) \cap \mathcal{N}_1(f)$ is an interval, which is non-empty if and only if $D(f) \leq 2$.

\newcommand{\regprime}{weakly regular\xspace}
\begin{definition}
  Let $f$ be a profile and suppose there exists a closed (possibly empty) subinterval $S$ of $\mathcal{N}_0(f) \cap \mathcal{N}_1(f)$  such that the following hold.
  \begin{enumerate}
  \item $f$ is $(m,\,M)$-regular on $[0,\,1] \setminus S$.
  \item $f$ is constant on $S$.
  \item If $D(f) \leq 2$, then neither endpoint of $\mathcal{N}_0(f) \cap \mathcal{N}_1(f)$ is an endpoint of $S$.
  \end{enumerate}
  Then $f$ is said to be $(m,\,M)$-\emph{\regprime}.
\end{definition}
\begin{prop}\label{lem:cont}
   The operator $U$ is continuous at any \regprime profile $f$, with respect to the norm $||\cdot||_\infty$. In particular, $U$ is continuous at any regular profile.
\end{prop}

\begin{taggedblock}{proofsprel}

\begin{proof}
This result was essentially proven as Proposition 4 in \cite{BHT2}, but in their formulation $f$ was assumed to be regular on all of $[0,\,1]$. We show that the proof goes through for this stronger formulation, which we will need later.

Let $f$ be a \regprime profile, with regularity bounds $m$ and $M$ on $[0,\, 1] \setminus S$.

Choose $0 < \delta \leq \frac{1}{4}$ such that, if $D(f) < 2$ then the distance between an endpoint of $\mathcal{N}_0(f) \cap \mathcal{N}_1(f)$ and an endpoint of $S$ is greater than $\frac{2\delta}{m}$.

We will show that, if $g$ is a profile such that $||f-g||_\infty \leq \delta$, then
\begin{align*}  
   ||Uf - Ug||_\infty \leq \frac{17M}{m}\delta.
\end{align*}

Fix such a profile $g$. Fix some agent $\alpha$ and define the following sets:
  \begin{align*}
    S_{fg} = & \mathcal{N}_\alpha(f) \cap \mathcal{N}_\alpha(g), \\
    S_{f \setminus g} = & \mathcal{N}_\alpha(f) \setminus \mathcal{N}_\alpha(g), \\
    S_{g \setminus f} = & \mathcal{N}_\alpha(g) \setminus \mathcal{N}_\alpha(f) .
  \end{align*}
  
  From Definition \ref{def:update1}, we get the following:

  \begin{align}
    Uf(\alpha) = & \langle f \rangle_{S_{fg}} + \frac{|S_{f \setminus g}|}{|S_{fg}| + |S_{f \setminus g}|} \left(\langle f \rangle_{S_{f \setminus g}} - \langle f \rangle_{S_{fg}} \right) \label{eq:contb} \\
    Ug(\alpha) = & \langle g \rangle_{S_{fg}} + \frac{|S_{g \setminus f}|}{|S_{fg}| + |S_{g \setminus f}|} \left(\langle g \rangle_{S_{g \setminus f}} - \langle g \rangle_{S_{fg}} \right). \label{eq:contc}
  \end{align}

  As both $S_{fg}$ and $S_{f \setminus g}$ are subsets of $\mathcal{N}_\alpha(f)$, the absolute value of the last parenthesis in \eqref{eq:contb} can be at most $2$, and the same holds for the parenthesis in \eqref{eq:contc}.
  Using the triangle inequality we get that
  \begin{equation}\label{eq:contd}
    |Uf(\alpha) - Ug(\alpha)|
    \leq
    |\langle f \rangle_{S_{fg}} - \langle g \rangle_{S_{fg}} |
    +
    2 \frac{|S_{g \setminus f}| + |S_{f \setminus g}|}{|S_{fg}|}
  \end{equation}

  Since $||f-g||_\infty \leq \delta$, it is obvious that $|\langle f \rangle_{S_{fg}} - \langle g \rangle_{S_{fg}} | \leq \delta$.

  We now note that the third condition for being \regprime and the definition of $\delta$ imply that neither of the sets $S_{g \setminus f}$ and $S_{f \setminus g}$ intersect $S$, and hence $f$ is regular on both. Thus $|S_{g \setminus f}|$ and $|S_{f \setminus g}|$ are each bounded by the measure of the set of agents which may be added to or removed from $\mathcal{N}_\alpha(f)$ by moving each agent at most $\delta$, which, by regularity, is at most $\frac{2\delta}{m}$. 

  Similarly, $S_{fg} \supseteq \{\beta : |f(\beta) - f(\alpha)| \leq 1 - 2\delta \}$. Hence $|S_{fg}| \geq \frac{1-2\delta}{M} \geq \frac{1}{2M}$, since $\delta \leq \frac{1}{4}$.

  The bounds from the previous paragraphs may be put into \eqref{eq:contd}, to give us
  \begin{equation*}
    |Uf(\alpha) - Ug(\alpha)|
    \leq
    \delta + \frac{16M \delta}{m}
    \leq
    17 \frac{M}{m} \delta.
  \end{equation*}
\end{proof}
\end{taggedblock}

\begin{lemma}\label{lem:derivatives}
  Let $f$ be a regular profile 
  and define 
  \begin{align*}
    u(\alpha) & =
    \begin{cases}
      0                  & \text{if $f(\alpha) \leq f(0)+1$}\\
      f^{-1}(f(\alpha)-1) & \text{otherwise,}
    \end{cases}\\
    v(\alpha) & =
    \begin{cases}
      1                  & \text{if $f(\alpha) \geq f(1)-1$}\\
      f^{-1}(f(\alpha)+1) & \text{otherwise,}
    \end{cases}\\
    w(\alpha) &= |\mathcal{N}_\alpha (f)| = v(\alpha) - u(\alpha).
  \end{align*}
  In words, $u(\alpha)$ and $v(\alpha)$ are the leftmost and rightmost agents, respectively, that interact with a given agent $\alpha$ when the profile is updated by $U$, and $w(\alpha)$ is the length of the set of neighbours of $\alpha$.

  Then the derivative of $Uf$, where it exists, is given by
  \begin{align}
    \frac{d}{d \alpha} Uf(\alpha) &= \frac{1}{w(\alpha)}
         ( u'(\alpha) \cdot \left(1 + Uf(\alpha) - f(\alpha) \right) +
           v'(\alpha) \cdot \left(1 + f(\alpha) - Uf(\alpha) \right) ),
  \end{align}
  where the primes denote derivatives.
\end{lemma}
\begin{proof}
See Lemma 2.5 in \cite{WH2}. The statement of that lemma assumes $D(f) > 2$, but the proof goes through even without this assumption.
\end{proof}

The following lemma was proved for regular profiles in \cite{BHT2} using a different technique, yielding weaker regularity bounds than those given here.
\begin{prop}\label{prop:regularity}
Let $f$ be an $(m,\,M)$-\regprime profile.

Then $U f$ is constant on the closed interval $\mathcal{N}_0(f) \cap \mathcal{N}_1(f)$ and $(\frac{m}{2 M^2},\, 2\frac{M^2}{m})$-regular on $[0,\,1] \setminus (\mathcal{N}_0(f) \cap \mathcal{N}_1(f))$.

In particular, if $f$ is weakly regular then $Uf$ is either weakly regular or a consensus, and if $f$ is regular and $D(f) \geq 2$ then $Uf$ is regular.
\end{prop}
\begin{taggedblock}{proofsprel}
\begin{proof}
First note that the second statement is a direct consequence of the first, so it suffices to prove the first statement.

It is clear that $Uf$ is constant on $\mathcal{N}_0(f) \cap \mathcal{N}_1(f)$, so for the rest of the proof we will assume  $D(f) > 1$ and only consider $\alpha \in [0,\,1] \setminus (\mathcal{N}_0(f) \cap \mathcal{N}_1(f)) := R$.

One readily verifies that almost everywhere differentiability on $R$ along with uniform upper and lower bounds on the derivative imply regularity with the same bounds. We prove the theorem by providing such bounds for $U f$.

  Since $f$ is monotone and regular on $R$, for almost every $\alpha \in R$ the derivatives $f'(\alpha)$, $f'(\alpha)$, $u'(\alpha)$, $u'(\alpha)$ and $w'(\alpha)$ all exist, by Lebesgue's theorem.
  
  For any such $\alpha$, by Lemma \ref{lem:derivatives},
  \begin{equation}\label{eq:regularitya}
    \frac{d}{d \alpha} Uf (\alpha)
    =\!\!
    \frac{1}{w(\alpha)} \!\!
    \left[
      u'(\alpha)\left( 1 + (Uf(\alpha) - f(\alpha))\right)
      +                                                   
      v'(\alpha)\left( 1 - (Uf(\alpha) - f(\alpha))\right)
      \right].
  \end{equation}

  We first prove the upper regularity bound for $Uf$.
  
  Applying the chain rule to $v$, we find that, if $f(\alpha) < f(1) - 1$ then
  \begin{align*}
    v'(\alpha) = \frac{d}{d\alpha} (f^{-1})(f(\alpha)+1) =
    \frac{f'(\alpha)}{f'(f(\alpha)+1)}
\leq \frac{M}{m},
\end{align*}
 where the inequality follows from the regularity bounds on $f$. If $\alpha$ is large enough for $v(\alpha)$ to be constantly $1$, the derivative is $0$, so the inequality remains true.
  In the same way we have $u'(\alpha) \leq \frac{M}{m}$.

As we assume $D(f) > 1$, $w(\alpha) \geq \frac{1}{M}$.

  We also have the trivial bound $|Uf(\alpha) - f(\alpha)| \leq 1$. Note also that the two parentheses 
$\left( 1 + (Uf(\alpha) - f(\alpha))\right)$ and 
$\left( 1 - (Uf(\alpha) - f(\alpha))\right)$ 
in \eqref{eq:regularitya} sum to $2$.

  Together, the observations from the previous paragraphs may be inserted into \eqref{eq:regularitya} to get that
  \begin{equation*}
    \frac{d}{d \alpha} Uf (\alpha) \leq M \left( 2 \frac{M}{m} \right) = 2 \frac{M^2}{m}.
  \end{equation*}


  As for the lower regularity bound, we first note the trivial bound $w(\alpha) \leq 1$ is the best we can do.

  Second, we note that, as we assume $\alpha \not\in \mathcal{N}_0(f) \cap \mathcal{N}_1(f)$, we cannot have $u'(\alpha)=v'(\alpha)=0$. We will here assume $v'(\alpha) \neq 0$, and note that the other case is completely analogous. Using the chain rule, as above, we see that
\begin{equation*}
  v'(\alpha) =
  \frac{d}{d\alpha} (f^{-1})(f(\alpha) + 1)=
  \frac{f'(\alpha)}{f'(f(\alpha)+1)}
  \geq
  \frac{m}{M}.
\end{equation*}

To finish the proof, it is enough to prove that
\begin{align}\label{eq:regularityd}
Uf(\alpha) - f(\alpha) \leq 1 - \frac{1}{2M}.
\end{align}

  Intuitively, to make $Uf(\alpha)-f(\alpha)$ as large as possible, we want to pack as many neighbours of $\alpha$ as far to the right as possible, while having as few neighbours as possible in the rest of the neighbourhood. 
  
  To formalise this, consider the profile $g$ such that
  \begin{equation}\label{eq:g}
    g(\gamma)=
    \begin{cases}
      M\gamma & \text{if } 0 \leq \gamma \leq \frac{1}{M}\\
      1 & \text{otherwise.}
    \end{cases}
  \end{equation}
  
  The assumptions on $f$ imply that
\begin{align}
Uf(\alpha) - f(\alpha) \leq Ug(0) - g(0) =
        Ug(0) = \frac{\frac{1}{2}\frac{1}{M} + \left(1 - \frac{1}{M} \right)\cdot 1}{1} = 1 - \frac{1}{2M},
\end{align}
which finishes the proof.
\end{proof}
\end{taggedblock}

\begin{rem}\label{rem:regularity}
For regular profiles, including any \regprime profile with diameter above 2, the lower regularity bound of Proposition \ref{prop:regularity} could be improved by exchanging the constant segment of the auxiliary profile $g$ in \eqref{eq:g} by a segment of slope $m$. We content ourselves with the current version as the extended proof is technical and the improvement is slight. When iterated, either version of the proposition results in the quotient $\frac{M_t}{m_t}$ asymptotically growing like $e^{\Theta (3^t)}$.
As these results will not be used we leave out the proofs.
\end{rem}

We are now ready to prove part (i) of Theorem \ref{thm:intro1} and deduce Corollary \ref{cor:intro2} from Theorem \ref{thm:intro1}.

\begin{proof}
For the deduction of Corollary \ref{cor:intro2}, it clearly suffices to prove the statement about the numbers $L_i'$.

Fix $L \in [0,\, L_1'] \cup \{ 6 \}$ and let $T$, as in Theorem \ref{thm:intro1},  be an upper bound on the freezing time for equally spaced profiles with diameter $L$.

Recall that $f^{n,L}$ denotes the canonical equally spaced profile on $n$ agents with diameter $L$ and that $f^L$ denotes the canonical linear profile with diameter $L$. Let $f_n^L$ denote the empirical quantile function (see \eqref{eq:quantile}) of a sample of size $n$ from the uniform distribution on $[0,\,L]$. The results in this section then give the following chain of implications:

{~}

\begin{tabular}{>{$}c<{$}  >{$}r<{$}  @{ } l}
 & U^T f^{n,L} & is a consensus for all $n \gg 0$\\
\stackrel{\ref{lem:cont},\, \ref{prop:regularity}}{\implies}
& D(U^{T} f^L) &  $< 1$      \\
\implies 
& U^{T+1} f^L & is a consensus    \\
\stackrel{\ref{lem:cont},\, \ref{prop:regularity}}{\implies}  & \exists \varepsilon > 0 : D(U^{T+1} (g     )) & $< 1$ if  $||f^L - g||_\infty < \varepsilon$.
\end{tabular}

~

From this last statement we deduce in turn the following.
\begin{itemize}
\item On the one hand, if $L'$ is sufficiently close to $L$, then $U^{T+2} f^{L'}$ is a consensus and, by a further application of Propositions \ref{lem:cont} and \ref{prop:regularity}, $U^{T+3} f^{n,L'}$ is a consensus for all $n$ sufficiently large. This proves part (i) of Theorem \ref{thm:intro1}.
\item On the other hand, by Corollary \ref{cor:glivcan}, $U^{T+2} f_n^L$ is a consensus a.a.s. This proves that Corollary \ref{cor:intro2} follows from Theorem \ref{thm:intro1}.
\end{itemize}
\end{proof}

By observing the proof just presented, it is clear that for some fixed $L$ the following three statements are equivalent:
\begin{enumerate}
\item There is some $T_1$ such that $U^{T_1} f^{n,L}$ is a consensus for all $n \gg 0$.
\item There is some $T_2$ such that $U^{T_2} f^{L}$ is a consensus.
\item There is some $T_3$ such that $U^{T_3} f_n^L$ is a consensus a.a.s. as $n \to \infty$.
\end{enumerate}

If we would have access to unlimited computing power, the theory developed this far would actually be enough to finish the proof of Theorem \ref{thm:intro1} in a few lines using the following strategy:

  Choose a really large $n$ so that $||f^{n,L} - f^L||_\infty < \delta$ for some $\delta$. Using Propositions \ref{lem:cont} and \ref{prop:regularity} we can compute constants $K_t$ for every $t\geq 0$ such that $||U^t f^{n,L} - U^t f^L||_\infty < K_t \delta$. We calculate the updates $U^t f^{n,L}$ explicitly. If we find that $U^{t_1} f^{n,L}$ is a consensus we check that $K_{t_1} \delta \leq \frac{1}{2}$, which must be true if $n$ is chosen large enough. We could then deduce that $U^{t_1+1} f^L$ is a consensus as well, and the rest would follow as above.

The problem with this strategy is that, as was hinted at in Remark \ref{rem:regularity}, the constant $K_t$ grows ridiculously fast as $t$ increases, so we would end up with $n$ needing to be much larger than can actually be simulated.
 Table 1 illustrates this.
\begin{table}\label{table:defs}
\begin{center}
\begin{tabular}{ c | c | c | c}
  $t$ & $m_t$ & $M_t$ & $K_t$\\
  \hline			
  0 & $1                  $& $ 1$               & 1\\        
  1 & $0.25               $& $ 2$               & 40\\                
  2 & $4.1 \cdot 10^{-3}  $ & $32$               & $3.9 \cdot 10^{5}$\\             
  3 & $2.9 \cdot 10^{-9}  $ & $5.0 \cdot 10^{5}$  & $8.6 \cdot 10^{14}$\\
  4 & $6.4 \cdot 10^{-31} $ & $1.7 \cdot 10^{20}$ & $1.4 \cdot 10^{51}$\\
  5 & $7.2  \cdot 10^{-107}$ & $9.4 \cdot 10^{70}$ & $6.5 \cdot 10^{177}$\\
  6 & $7.7  \cdot 10^{-373}$ & $2.4\cdot 10^{248}  $& $1.6  \cdot 10^{621}$\\
\end{tabular}
\end{center}
\caption{The constants grow too fast to be of practical use.}
\end{table}

To get around this, we introduce a fourth statement.
\begin{enumerate}
\item[(iv)] There is some $T_4$ such that $U^{T_4} f^{n_i,L}$ is a consensus for some infinite sequence $n_1 < n_2 < n_3 < ...$.
\end{enumerate}
It is straightforward to check that this is also equivalent to the earlier three, and we will devote Sections 4 and 5 to prove (iv). The following corollary of Proposition \ref{prop:regularity} will be used in both sections.

\begin{corollary}\label{cor:regularity}
  Let $f$ be a symmetric regular profile. If $D(U^{t_1} f) \leq 2$ for some $t_1$, then $U^{t_2} f$ is a consensus for some $t_2$. Moreover, $t_2$ depends only on $t_1$ and the regularity bounds for $f$.
\end{corollary}
\begin{proof}
Without loss of generality, suppose $f$ is symmetric about 0. By Observation \ref{obs:symmetry}, the same is true of $U^t f$ for any $t \geq 0$.

If $D(f) \geq 2$ then Proposition \ref{prop:regularity} tells us that $U f$ is regular. By iterating this we see that either $D(U^t f) > 2$ for all $t$, in which case we are done, or there is some first time $t_1$ such that $D(U^{t_1} f) \leq 2$, in which case $U^{t_1} f$ is regular.

Set $t_3 =t_1$ if $D(U^{t_1} f) < 2$ or $t_3 = t_1 + 1$ if $D(U^{t_1} f) = 2$.

Then, by Proposition \ref{prop:regularity}, $U^{t_3} f$ is still regular and, clearly, \newline $D(U^{t_3}f)<~2$.

By regularity and Proposition \ref{prop:regularity}, there is some $M_{t_3}$, depending only on $t_3$ and the regularity bounds for $f$, such that
\begin{equation*}
| \mathcal{N}_0(U^{t_3} f) \cap \mathcal{N}_1(U^{t_3} f) | \geq \frac{2 - D(U^{t_3} f)}{M_{t_3}} > 0.
\end{equation*}
By symmetry, $U^{t_3+1} f$ must be constantly equal to 0 on $\mathcal{N}_0(U^{t_3} f) \cap \mathcal{N}_1(U^{t_3} f)$. In fact, it is easy to see that $U^t f$ must be constantly equal to 0 on this interval for \emph{any} $t > t_3$. Hence one can check that, as long as the diameter is above 1, each extremist must change its opinion by at least $ \frac{2 - D(U^{t_3} f)}{2 M_{t_3}}$ at each time step. Thus the diameter must be at most 1 after at most $\frac{2 M_{t_3}}{2 - D(U^{t_3} f)}$ additional time steps.

This finishes the proof with $t_2 = t_3 + \frac{2 M_{t_3}}{2 - D(U^{t_3} f)} + 1$.

\end{proof}

\section{Propagation of errors due to refinements}

In the previous section we investigated the updating operator $U$ and, in particular, we noted that it is continuous at regular profiles. As we saw, the continuity by itself is not very helpful. In this section, we will shift our focus away from regular profiles back to discrete ones. Specifically, we will investigate the updates of profiles that have been perturbed, by movement or refinement, and derive bounds for the difference between these and the updates of the unperturbed profiles. All the profiles in this section are discrete.

As we have seen, by definition, if $f$ is a profile with $n$ agents, and $g$ is a $k$-regular refinement of $f$, we have $f = B_k^n(g)$, and thus $Uf = U B_k^n(g)$. We will begin by comparing $U B_k^n(g)$ to $B_k^n(Ug)$. Thus, one could, informally, say that the following lemma bounds the commutator of the two operators $U$ and $B_k^n$.

\begin{lemma}\label{lem:ghostagents}
  Let $f$ be a discrete profile with $n$ agents, and let $g$ be any $k$-regular refinement of $f$. For any agent $i \in [n]$
  \begin{equation}
    | U f(i) - B_k^n(Ug)(i)| \leq \frac{2}{\#(\mathcal{N}_i(f))}.
  \end{equation}
\end{lemma}

\begin{taggedblock}{proofslemmas}
  \begin{proof}
  Fix an agent $i \in [n]$. We will proceed by constructing a $k$-regular refinement $g^*$ of $f$ which maximises $B_k^n(Ug^*)(i)$, in the sense that $B_k^n(Ug)(i) \leq B_k^n(Ug^*)(i)$ for \emph{any} $k$-regular refinement $g$ of $f$.
  
  It is clear that, to maximise $B_k^n(Ug^*)(i)$, one may simply begin with $f$ and place all inserted opinions at the rightmost end of their interval, except for those in the interval containing the opinion $f(i) + 1$ who are placed there, and those in the interval immediately to the left of the leftmost neighbour of $i$ who are placed out of sight, i.e. below $f(i)-1$. We observe that, following this procedure, $B_k^n (U g^*) (i)$ is increasing with $k$ and is bounded by what is obtained if one changes the opinion of the leftmost neighbour of $i$ to $f(i) + 1$ before updating $f$.
  In other words,
  \begin{equation*}
    B_k^n (U g)(i) 
    \leq U f^+_i(i),
  \end{equation*}
  where $f^+_i$ denotes the profile obtained from taking $f$ and moving the leftmost neighbour of $i$ to $f(i)+1$.

  Now, note that moving one out of at least $n_i := \#(\mathcal{N}_i(f))$ opinions a distance at most $2$ cannot affect the updated opinion by more than $\frac{2}{{{n}}_i}$.

  We finally note that the reasoning is completely analogous for finding a lower bound for $B_k^n (Ug) (i)$.
  \end{proof}
\end{taggedblock}

The difference $e = g - f$ between two discrete profiles $f$ and $g$ with the same number $n$ of agents is a pre-profile with $n$ agents. On the other hand, for an arbitrary pre-profile $e$, $f+e$ need not be a profile. 
 In what follows we will use the term \emph{deviation} instead of pre-profile when thinking in terms of $e$ as a small perturbation of a given profile $f$. We will adopt the same shorthand for deviations as for profiles, and write $e(i)$ instead of $e \left(\frac{i}{n}\right)$ when there is no risk of confusion.

\begin{definition}  
  A deviation $e$ is called \emph{consistent} with respect to a discrete profile $f$ on $n$ agents if $f+e$ is a profile, 
  i.e. if
  \begin{equation*}
    f(i-1) + e(i-1) \leq f(i) + e(i) 
  \end{equation*}
  for all $i \in 2,\, ...,\, n-1$.
  
  For a deviation $e$ on $n$ agents, we will refer to positive deviations $e_l$ and $e_r$ on $n$ agents as \emph{left} and \emph{right bounds} on $e$, respectively, if they satisfy
  \begin{equation*}
    f(i) - e_l(i) \leq f(i) + e(i) \leq f(i) + e_r(i)
  \end{equation*}
  for any $i \in [n]$.
\end{definition}

For any profile $f$ on $n$ agents and any $k,\, t\in \mathbb{N}$ we get
\begin{align*}
  B_k^n(U^t f^{(k)}) = U^t f + e_k^t
\end{align*}
where $e_k^t$ is clearly a consistent deviation. By Lemma \ref{lem:ghostagents}, $||e_k^1||_\infty$ is uniformly bounded in $k$. If we want to iteratively obtain bounds for $e_k^t$, we need to compare $Ug$ to $U(g+e)$ for generic $g$ and $e$. We will not aim for bounds in  $|| \cdot ||_\infty$, instead our bounds will depend on the agent $i$. However, the bounds will still be uniform in $k$ for each $t$, which is the crucial point.

 Adding a deviation to a profile may cause the neighbourhoods of its agents to change and we start by introducing some notation to handle these changes.

\begin{definition}
  Given a profile $f$ on $n$ agents and a deviation with bounds $e_l$ and $e_r$, we define the sets
  \mathtoolsset{showonlyrefs=false} 
  \begin{align}
    \vec{\mathcal{N}}^+_i(f) =& \{ j > i : f(i)+1 < f(j) \leq f(i)+1 + e_r(i) + e_l(j) \}\label{eq:narrowfirst} \\
    \vec{\mathcal{N}}^-_i(f) =& \{ j < i : f(i)-1 \leq f(j) < f(i)-1 + e_r(i) + e_l(j) \} \\
    \cev{\mathcal{N}}^+_i(f) =& \{ j < i : f(i)-1 - e_l(i) - e_r(j) \leq f(i) < f(i)-1 \} \\
    \cev{\mathcal{N}}^-_i(f) =& \{ j > i : f(i)+1 - e_l(i) - e_r(j) < f(j) \leq f(i)+1 \}.\label{eq:narrowlast} 
  \end{align}
  \mathtoolsset{showonlyrefs=true} 

  In words, $\vec{\mathcal{N}}_i^+(f)$ contains precisely the agents that may be added on the right side of the neighbourhood of $i$ by a perturbation bounded by $e_l$ and $e_r$, and $\vec{\mathcal{N}}_i^-(f)$ contains those that may be removed on the left. The right arrow above the $\mathcal{N}$ indicates that the average opinion of $i$'s neighbours is increased. The sets denoted with left arrows are defined analogously.
\end{definition}

\begin{lemma}\label{thm:devprop}
  Let $f$ be a profile with $n$ agents, and let $e$ be a consistent deviation bounded by some $e_l$ and $e_r$. Then
  \begin{align}\label{eq:devprop}
    U(f+e)(i) \leq &
    \langle f(j) + e_r(j) \rangle_{j \in \mathcal{N}_i(f) \setminus \vec{\mathcal{N}}^-_i(f)} +
    \frac{2 |\vec{\mathcal{N}}^+_i(f)|}{|\mathcal{N}_i(f)| + |\vec{\mathcal{N}}^+_i(f)| - |\vec{\mathcal{N}}^-_i(f)|} ,\\
    U(f+e)(i) \geq &
    \langle f(j) - e_l(j) \rangle_{j \in \mathcal{N}_i(f) \setminus \cev{\mathcal{N}}^-_i(f)} -
    \frac{2 |\cev{\mathcal{N}}^+_i(f)|}{|\mathcal{N}_i(f)| + |\cev{\mathcal{N}}^+_i(f)| - |\cev{\mathcal{N}}^-_i(f)|} .
  \end{align}
\end{lemma}

\begin{taggedblock}{proofslemmas}
  \begin{proof}
    As the proofs of the two inequalities are completely analogous, we only treat the first.
    
  Fix a profile $f$, bounds $e_l$ and $e_r$, and an agent $i \in [n]$.

  We will bound the maximum $\max_e U(f+e)(i)$ over all consistent deviations $e$ within the given bounds by constructing a ``worst case scenario'' $e^*$, where all agents work to push the opinion of $i$ as far to the right as possible.

  Before continuing we recall that, since we assume consistency, $e$ does not change the order of the agents. This is important: If the leftmost neighbour of $i$ is not moved out of sight of $i$ it remains the leftmost neighbour, and thus its opinion cannot exceed the average opinion of the remaining neighbours. We thus want any left-removable neighbour to be removed, and each addable neighbour to the right to indeed end up so.

  We thus proceed as follows.
  \begin{itemize}
  \item For each agent $j \in \vec{\mathcal{N}}^-_i(f)$, let $e^*(j) = e_l(j)$.
  \item For each agent $j \in \vec{\mathcal{N}}^+_i(f)$, choose $e^*(j)$ so that $f(j) + e^*(j) = f(i) + e_r(i) + 1$. 
  \item For each agent $j \in \mathcal{N}_i(f) \setminus \vec{\mathcal{N}}^-_i(f)$, 
    choose $e^*(j)$ so that
    \begin{align*}
      f(j) + e^*(j) = \min \{ f(j) + e_r(j),\, f(i) + e_r(i) + 1 \}.
    \end{align*}
  \item For each other $j$, choose $e^*(j)$ so that $j$ ends up out of sight of $i$.
  \end{itemize}
  
  For a deviation $e^*$ constructed in this way we have
  \begin{align*}
    \pushleft{U(f+e^*) (i) = } &
    \\ \frac{\langle f(j) + e^*(j)\rangle_{\mathcal{N}_i(f) \setminus \vec{\mathcal{N}}^-_i(f)} (|\mathcal{N}_i(f)| - | \vec{\mathcal{N}}^-_i(f)|) + (f(i) + e_r(i) + 1) | \vec{\mathcal{N}}^+_i(f) |
    }{ |\mathcal{N}_i(f)| + |\vec{\mathcal{N}}^+_i(f)| - |\vec{\mathcal{N}}^-_i(f)| } & \leq \\
    \frac{\langle f(j) + e_r(j)\rangle_{\mathcal{N}_i(f) \setminus \vec{\mathcal{N}}^-_i(f)} (|\mathcal{N}_i(f)| - | \vec{\mathcal{N}}^-_i(f)|) + (f(i) + e_r(i) + 1) | \vec{\mathcal{N}}^+_i(f) |
}{ |\mathcal{N}_i(f)| + |\vec{\mathcal{N}}^+_i(f)| - |\vec{\mathcal{N}}^-_i(f)| }. &
  \end{align*}
  Let $\bar{f} = \langle f(j) + e_r(j)\rangle_{j \in \mathcal{N}_i(f) \setminus \vec{\mathcal{N}}^-_i(f)}$, and note that, by definition of the set $\vec{\mathcal{N}}^-_i(f)$, for any $j \in \mathcal{N}_i(f) \setminus \vec{\mathcal{N}}^-_i(f)$ one has
  \begin{align*}
    f(i) + e_r(i) \leq 1 + f(j) - e_l(j) \leq 1 + f(j) + e_r(j).
  \end{align*}
  Hence $f(i) + e_r(i) + 1 \leq \bar{f} + 2$ and thus
  \begin{align*}
    U(f+e^*)(i) &\leq
    \frac{\bar{f}(|\mathcal{N}_i(f) \setminus \vec{\mathcal{N}}_i^-(f)|) + \left(\bar{f} +2 \right) |\vec{\mathcal{N}}_i^+(f)|
    }{
      |\mathcal{N}_i(f)| + |\vec{\mathcal{N}}_i^+(f)| - |\vec{\mathcal{N}}_i^-(f)|}\\
    & = \bar{f} + \frac{2 |\vec{\mathcal{N}}_i^+(f)|
    }{
      |\mathcal{N}_i(f)| + |\vec{\mathcal{N}}_i^+(f)| - |\vec{\mathcal{N}}_i^-(f)|},
  \end{align*}
  which proves the inequality \eqref{eq:devprop}.
\end{proof}
\end{taggedblock}

\begin{theorem}\label{cor:devprop}
  Let $f$ be a discrete profile with $n$ agents, and let $e = e^0$ be a consistent deviation of $f$ with right and left bounds $e_r^0$ and $e_l^0$.

  Now, let $g$ be a $k$-regular refinement of $f$, and let $e'$ be a consistent deviation of $g$ such that $B_k^n(e') = e$.

  Then, bounds for the deviation $e^T = B_k^n(U^T(g+e'))-U^Tf$ may be found iteratively by defining
  \begin{align}
    e_r^{t+1}(i) = &
    \langle U^tf(j) + e_r^t(j) \rangle_{j \in \mathcal{N}_i(U^tf) \setminus \vec{\mathcal{N}}^-_i(U^tf)} +\\
    & + \frac{2 |\vec{\mathcal{N}}^+_i(U^tf)|}{|\mathcal{N}_i(U^tf)| + |\vec{\mathcal{N}}^+_i(U^tf)| - |\vec{\mathcal{N}}^-_i(U^tf)|} 
     - U^{t+1} f(i) + \frac{2}{{{n}}_i^t}\\ \label{eq:rightbound}
    &\\
    e_l^{t+1}(i) = & 
    - \langle U^tf(j) - e_l^t(j) \rangle_{j \in \mathcal{N}_i(U^tf) \setminus \cev{\mathcal{N}}^-_i(U^tf)} -\\
    & {} - \frac{2 |\cev{\mathcal{N}}^+_i(U^tf)|}{|\mathcal{N}_i(U^tf)| + |\cev{\mathcal{N}}^+_i(U^tf)| - |\cev{\mathcal{N}}^-_i(U^tf)|} + U^{t+1} f(i) + \frac{2}{{{n}}_i^t},\\ \label{eq:leftbound}
  \end{align}

  where all sets $\cev{\mathcal{N}}^+_i(U^tf)$, $\cev{\mathcal{N}}^-_i(U^tf)$, $\vec{\mathcal{N}}^+_i(U^tf)$ and $\vec{\mathcal{N}}^-_i(U^tf)$
  are defined as in \eqref{eq:narrowfirst}--\eqref{eq:narrowlast} using the bounds $e_l^t$ and $e_r^t$,
  and
  \begin{equation*}
  {{n}}_i^t = |\mathcal{N}_i(U^t f)| - |\vec{\mathcal{N}}^-_i(U^t f)| - |\cev{\mathcal{N}}^-_i(U^t f)|
  \end{equation*}
  is a lower bound on the number of neighbours of an agent in
    ${B_k^n (U^t (g+e'))}$.  
\end{theorem}

\begin{taggedblock}{proofslemmas}
  \begin{proof}
    The proof is an immediate application of Lemmas \ref{lem:ghostagents} and \ref{thm:devprop}.
\end{proof}
\end{taggedblock}

\section{The case $L \leq 5.2$}
\newcommand{\diam}[2]{\text{span}_{#1}(#2)} 

In this section we will prove part of Theorem \ref{thm:intro1}. Because of the equivalences (i) and (ii) at the end of Section 2, it suffices to show that there exists a $T$ such that, for $L \in [0,\, 5.2]$, $U^T f^L$ is a consensus. 
    
  Fix $0 < \varepsilon < \frac{1}{2}$ and an odd number $n \geq 3$. Let $\left( f^{n,L_j} \right)_{j=1}^p$ denote a finite family of canonical equally spaced profiles, all with $n$ agents and increasing diameters $L_j$ chosen such that $L_{j+1} - L_j \leq 2\varepsilon$ for any $j \in [p-1]$. For each such $f^{n,L_j}$, define the following left and right deviation bounds:
\begin{align*}
  e_{l,j}^0(i) =
  e_{l}^0(i) &=
  \begin{cases}
    \frac{n + 1 - 2i}{n-1}\varepsilon & \text{if } 1 \leq i < \frac{n+1}{2} \\
    0  & \text{if } \frac{n+1}{2} \leq i \leq n
  \end{cases}\\
    e_{r,j}^0(i) = 
  e_r^0(i) &= 
  \begin{cases}
    0  & \text{if } 1 \leq i \leq \frac{n+1}{2}\\
    \frac{2i - n - 1}{n-1}\varepsilon & \text{if } \frac{n+1}{2} < i \leq n .
  \end{cases}
\end{align*}
Note that $e_l^0 (\frac{n+1}{2}) = e_r^0 (\frac{n+1}{2}) = 0$ and $e_l^0(1) = e_r^0(n) = \varepsilon$. 

\begin{prop}\label{thm:grid}
  With $\left (f^{n,L_j} \right)_{j=1}^p$, $e_{l,j}^0$ and $e_{r,j}^0$ defined as above, assume there is a time $T>0$ such that for each $j \in [p]$
  \begin{align} \label{eq:grid}
    D(U^{T} f^{n,L_j}) + e_{l,j}^{T}(1) + e_{r,j}^{T}(n) < 2,
  \end{align}
 where $e_{l,j}^{T}$ and $e_{r,j}^{T}$ are determined recursively as in Theorem \ref{cor:devprop}.
  
  Then there is a time $T'$ such that $U^{T'} f^L$ is a consensus for every $L \in [L_1,\, L_p]$.
\end{prop}
\begin{proof}
  The definitions of $e_l^0$ and $e_r^0$ imply that if, for every $L \in [L_1,\, L_p]$, the profile $f^{n,L}$ is shifted to the interval $[-\frac{L}{2},\, \frac{L}{2}]$, then for each $L$ there exists a $j \in [p]$ such that $f^{n,L_j}$ can be transformed into $f^{n,L}$ by adding a deviation within these bounds.
  
  Hence, Theorem \ref{cor:devprop} and \eqref{eq:grid} imply that there exists $\delta > 0$ such that, for any $k \geq 0$ and any $L \in [L_1,\, L_p]$, one has $D(U^T f^{n + (n-1)k,L}) < 2 - 2 \delta$. Then Proposition \ref{lem:cont} implies that $D(U^T f^L) < 2$ for any $L \in [L_1,\, L_p]$. Hence the proposition follows from Corollary \ref{cor:regularity}, since $L \in [2,\,5.2]$ and so we can choose regularity bounds for $f^L$ which are independent of $L$.
\end{proof}

\begin{table}[ht!]\label{tbl:caseL5}
    \begin{center}
      \begin{tabular}{l | >{$} c <{$} | c | c | c | c }
        $k$ & [a_k,\, b_k] & $p_k$ & $n_k$ & $T_k$ & Runtime\\
        \hline \hline
        1 & [2,\, \frac{23}{6}]               & 100       & 501    & 3  &  4.1 seconds\\ \hline
        2 & [\frac{23}{6},\,  4.5]            & 100       & 1003    & 4  & 6.6 seconds \\ \hline
        3 & [4.5,\,    4.9]                   & 800       & 5001    & 5  & 109 seconds \\ \hline
        4 & [4.9,\,    5]                     & 200       & 10~001  & 6  &  66 seconds  \\ \hline
        5 & [5,\,      5.1]                   & 800       & 40~001  & 6  &  19 minutes  \\ \hline
        6 & [5.1,\,    5.15]                  & 1000      & 60~001  & 7  &  41 minutes  \\ \hline
        7 & [5.15,\,   5.17]                  & 480       & 120~001 & 7  &  41 minutes  \\ \hline
        8 & [5.17,\,   5.185]                 & 500       & 180~001 & 7  &  67 minutes  \\ \hline
        9 & [5.185,\,  5.193]                 & 420       & 240~001 & 7  &  97 minutes  \\ \hline
       10 & [5.193,\,  5.197]                 & 480       & 240~001 & 7  & 108 minutes  \\ \hline
       11 & [5.197,\,  5.2]                   & 360       & 300~001 & 7  & 102 minutes \\ \hline
      \end{tabular}
      \caption{The parameters used to check \eqref{eq:caseL5} and thus prove Corollary \ref{cor:caseL5}. The runtime given is the time it took us to complete the calculations using 40 threads running in parallel. With a total CPU time of almost 270 hours some sort of cluster is thus necessary to check the full proof, but checking it up to $L = 5$ should be feasible on any modern computer.}
    \end{center}
  \end{table}

Proposition \ref{thm:grid} reduces the proof of Theorem \ref{thm:intro1} for $L \in [0,\, 5.2]$ to a finite computation.

\begin{corollary}\label{cor:caseL5}
  There is a time $T$ such that $U^T f^L$ is a consensus for every $L \leq 5.2$.
\end{corollary}
\begin{proof}
  That such a $T$ exists for $L \leq 2$ follows from\footnote{It's trivial that $U f^L$ is a consensus for any $L \leq 1$ and it is easy to show that \newline $L \leq 2 \implies D(U f^L) \leq 1 \implies U^2 f^L$ is a consensus.} 
  Corollary \ref{cor:regularity}. Hence, we may henceforth assume that\footnote{In fact, for $L \leq \frac{23}{6}$, Corollary \ref{cor:caseL5} follows from the explicit formula for $U^2 f^{n,L}$ in Appendix A. See Remark \ref{rem:23over6}. We start our computations at $L = 2$ anyway, as it only adds a couple of seconds to the runtime.}
 $L \in [2,\, 5.2]$.

  We divide the interval $[2,\, 5.2]$ into smaller subintervals $[a_k,\, b_k]$ and apply Proposition \ref{thm:grid} to the individual subintervals. For each subinterval $[a_k,\, b_k]$ we choose numbers $p_k$, $n_k$, and let the diameters $\left( L_j \right)_{j=1}^{p_k}$ be spread equidistantly over the subinterval so that $L_1 = a_k$ and $L_{p_k} = b_k$. We use $\varepsilon = \frac{b_k - a_k}{p_k - 1}$.

  For each $k$ we then use the Julia code in the ancillary file \newline \texttt{L\_up\_to\_5\_2\_parallel.jl} to find a $T_k$ such that, for each $j \in [p_k]$ we have
  \begin{align}\label{eq:caseL5}
    D(U^{T_k} f^{n_k,L_j}) + e_{l,j}^{T_k}(1) + e_{r,j}^{T_k}(n_k) < 2.
  \end{align}
  The data is presented in Table 2, which proves the corollary.
\end{proof}

\section{The case $L = 6$}
Because of the discussion at the end of Section 2, in order to complete the proof of Theorem \ref{thm:intro1}, it suffices to show that there is an $n$ and some bounded time $T$ such that $U^T f^{n+(n-1)k,6}$ is a consensus for any $k \in \mathbb{N}$.

One could imagine trying the approach of the previous section, to simply choose some large $n$ and calculate consecutive updates $U^t f^{n,6}$ until the diameter shrinks below 2, thereafter applying Corollary \ref{cor:regularity}. However, as observed by Lorenz and already mentioned in the introduction, the evolution in this case passes through a kind of ``quasi-stable'' state where most agents are clustered into 5 groups, after which the diameter decreases very slowly. Our computations suggest that for large $n$ we would have to update approximately 780 times before the diameter goes below 2. Over such time scales the error analysis from Section 3 becomes impractical on its own.

Another approach is necessary.

\setcounter{figure}{0}

\begin{figure}[ht!]\label{fig:a}
\includegraphics[width=1\textwidth]{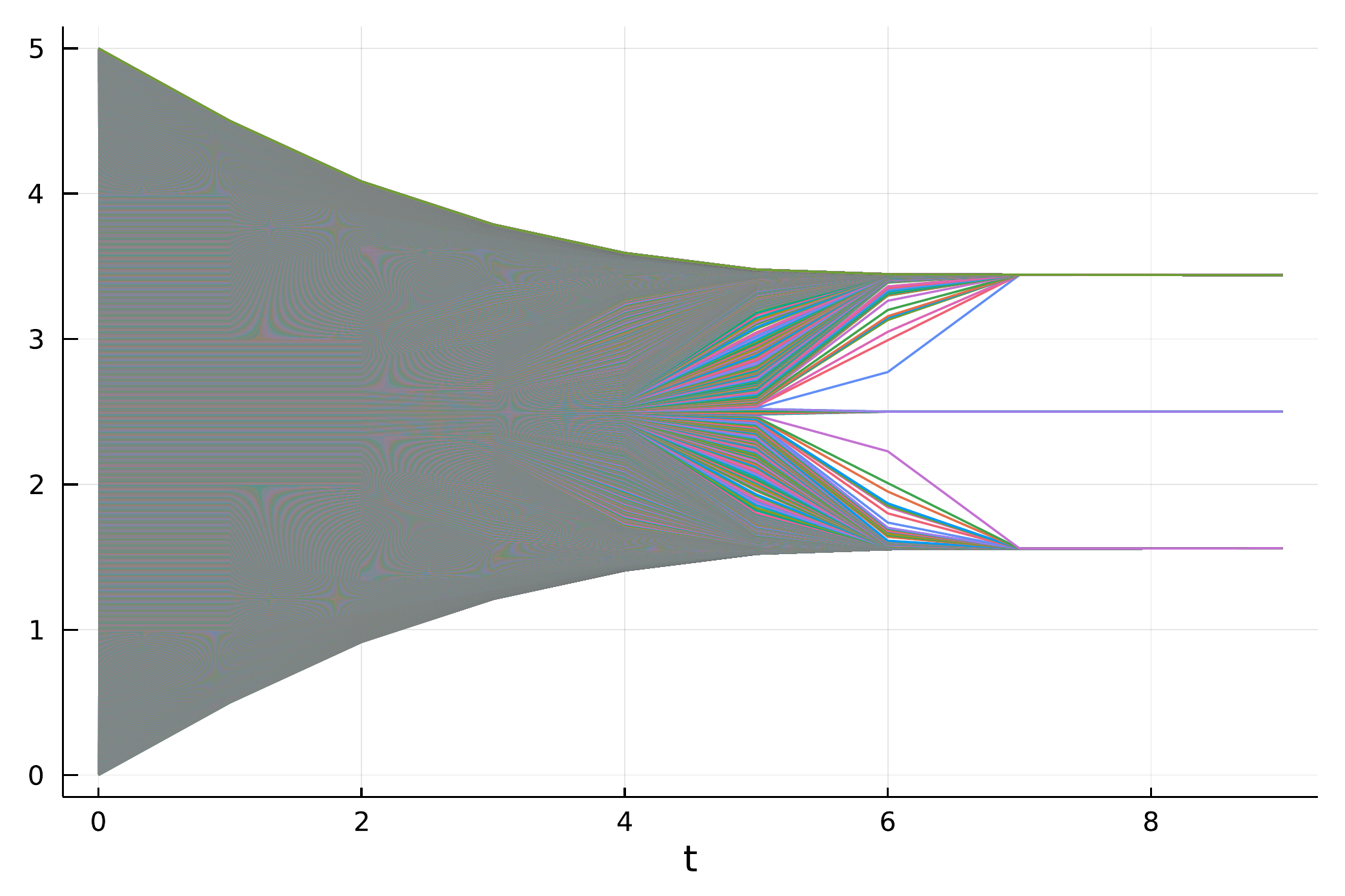}
\includegraphics[width=1\textwidth]{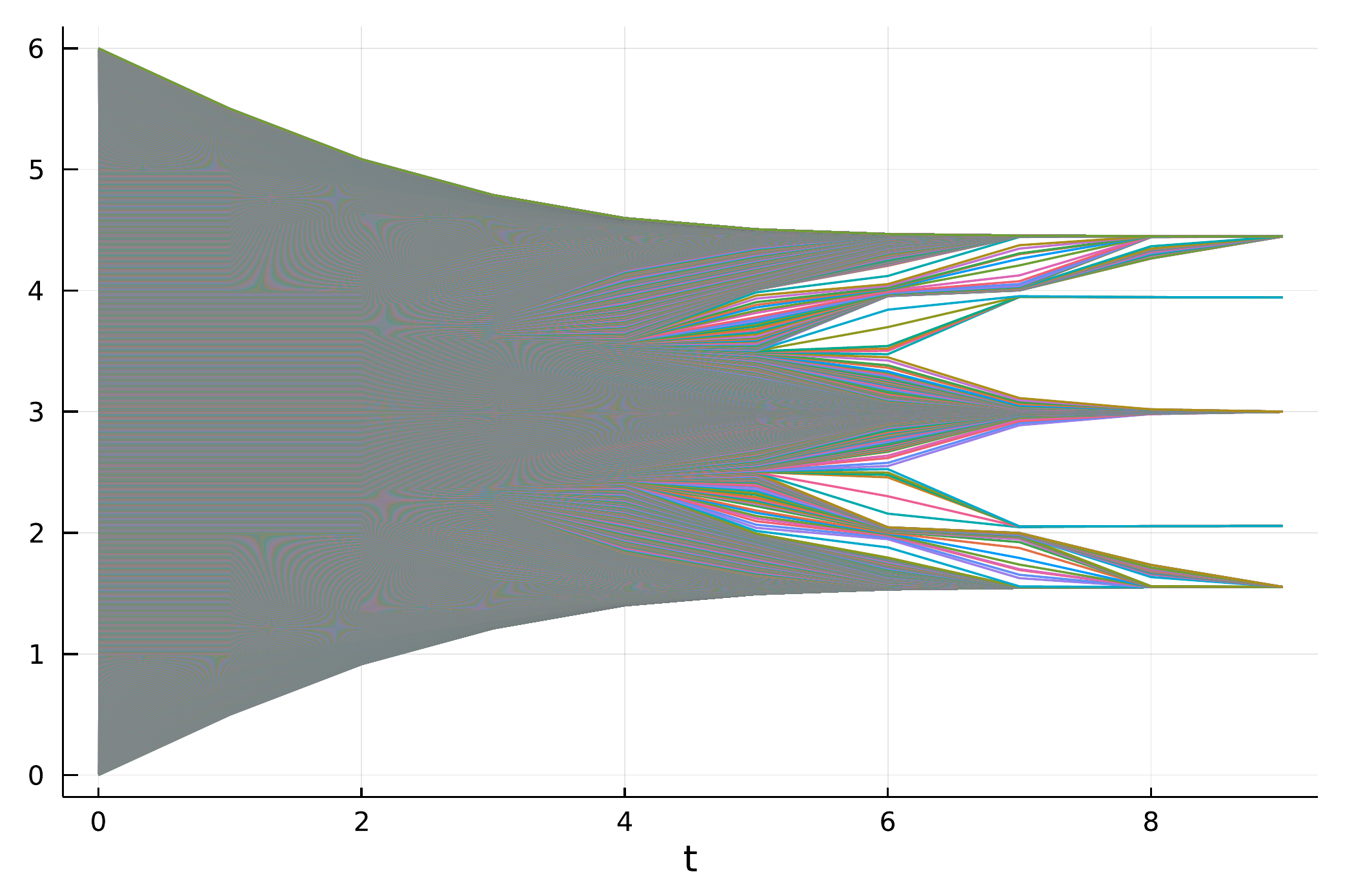}
\caption{First 9 updates of $f^{n,L}$ with $L = 5$, respectively $L = 6$ and $n = 10~000$.
}
\end{figure}

The figure below compares the evolution of two profiles $f^{n,5}$ and $f^{n,6}$ for $n = 10~000$. As discussed in the previous paragraph, one finds that the respective mechanisms behind reaching a consensus are quite different. In both profiles the agents at the edges first move inwards, so that the agent-per-opinion density increases at the edges. This causes the agents further in to move outwards, leaving comparatively empty spaces which in turn make the remaining agents clump together at the centre.

When $L = 5$ it only takes five updates for the extremists to move enough to see past the centre (see Table 2: there $T_{4} = 6$ because it takes one further timestep for \eqref{eq:grid} to hold) so even though the central clique is small and it takes many updates to reach a consensus, we may use Corollary \ref{cor:regularity} to assert consensus in bounded time. 

When $L = 6$, the clique of agents at the centre is larger but the profile is so wide that this clique is far out of sight of the extremists. However, there are a few agents that seem to get stuck in between. We will refer to the set of in-between agents as a \emph{microcluster}. The agents in the microcluster see both the central and extremal groups, but as the groups are of similar size the updates keep the microcluster agents within sight of both. Thus the microclusters will, so to speak, exert a tiny force on the extremists, and very slowly pull them towards the centre. It takes several hundred updates even before the extremists see the centre.

We will now formalise the content of the previous paragraphs and find a criterion for profiles which guarantees consensus (Theorem \ref{thm:6tocons}). We will then use the error analysis from Section 3 to show that
\newline
$U^8 f^{n+(n-1)k,6}$ satisfies this criterion for $n = 80 ~ 005$ and any $k \geq 0$ (having $n \equiv 1 ~ (\text{mod }6)$ simplifies the code used).

\begin{definition}\label{def:good}
Let $f$ be a symmetric, discrete profile on $n \geq 3$ agents with $D(f) \leq 4$. A family $A$, $B$, $C$, $D$ and $E$ of of subsets of $[n]$ is said to be \emph{good} if the following conditions hold:
\begin{enumerate}
\item Each of $A$, $B$, $C$, $D$, $E$ is an interval.
\item The intervals are pairwise disjoint and adjacent, by which we mean $\min B = \max A + 1$ etc.
\item $1 \in A$.
\item $E$ is non-empty and $\max E + \min E = n + 1$, i.e. $E$ is symmetric about the midpoint of the profile.
\item $C = \{ i \in B \cup C \cup D: A \cup E \subseteq \mathcal{N}_i(f) \}$
\end{enumerate}
In words, $C$ consists of the agents in between $A$ and $E$ who in $f$ see both in their entirety, while $B$ and $D$ contain the agents left over on either side of $C$.
\end{definition}

In what follows, we will denote $L_t = D(U^t f)$ and, for a set $X$ of agents, write $\diam{t}{X} := U^t f(\max X) - U^t f(\min X)$ to denote the span of opinions within $X$ at time $t$. The size of a set $X$ of agents will be denoted $|X|$.

We are now ready to state the new consensus criterion:

\newcommand{\tzero}{0}
\begin{theorem}\label{thm:6tocons}
Let $f$ be a symmetric discrete profile on $n \geq 3$ agents, where $n$ is odd, with $D(f) \leq 4$.

Assume there is a choice of good sets $A_{\tzero}$, $B_{\tzero}$, $C_{\tzero}$, $D_{\tzero}$, $E_{\tzero}$ such that $A_0$ and $E_0$ are out of sight of one another, i.e. $f(\max A_0) + 1 < f(\min E_0)$, and such that the following conditions hold:

\begin{align}\label{eq:6tocons1}
C_0 \neq \emptyset.
\end{align}

\begin{align}
\frac{n-2|A_{0}|}{n-|A_{0}|} \frac{L_{0}}{2} \leq 1 \label{eq:6tocons4}
\end{align}

\begin{align}
\frac{n-|E_{0}|}{n + |E_{0}|} \frac{L_{0}}{2} + \frac{2 |B_{0}|}{n - 2|A_{0}| - |B_{0}|} \leq 1  \label{eq:6tocons5}
\end{align}

\begin{align}
\frac{2 |B_{0}|}{n - 2|A_{0}| - |B_{0}|} + \frac{4 |D_{0}|}{n - |E_{0}|} \leq  \frac{2 |C_{0}||E_{0}|}{(|A_{0}|+|B_{0}|+|C_{0}|)(n+|E_{0}|)} \label{eq:6tocons6}.~~~\,
\end{align}

Then $U^T f$ is a consensus for some $T$ which only depends on $L_0$ and the relative sizes $\frac{|A_0|}{n}$, $\frac{|C_0|}{n}$ and $\frac{|E_0|}{n}$ of the non-empty sets $A_0$, $C_0$ and $E_0$.
\end{theorem}
\begin{rem}
The set $C_0$ plays the role of what we in the beginning of this section called a microcluster. The idea of the proof is to show that this is large enough and exerts enough force to eventually bring the extremists in $A_0$ into contact with the central agents in $E_0$ before either set fragments and thereby reach a consensus.
\end{rem}

\begin{taggedblock}{proofsuniform}
\begin{proof}[Proof of Theorem \ref{thm:6tocons}]
 Recursively, for every $t>0$, set
\begin{align*}
A_t &:= A_0\\
E_t &:= E_0\\
C_t &:= \{ i \in B_0 \cup C_0 \cup D_0: A_{0} \cup E_{0} \subseteq \mathcal{N}_i(U^tf) \}\\
B_t &:= B_{0} \setminus C_t\\
D_t &:= D_{0} \setminus C_t.
\end{align*}

We divide the proof into three steps.

~~

{\bf Step 1:}
For each $t \geq 0$, define the following conditions.
\begin{enumerate}[label = \Roman*${}_t$:]
\item $C_{\tzero} \subseteq C_{1} \subseteq ... \subseteq C_t$.\label{i1}
\item $\diam{t}{A_0} \leq \frac{4 |D_{\tzero}|}{n - |E_{\tzero}|} := w_A$ and  $\diam{t}{E_0} \leq \frac{4 |B_{\tzero}|}{n - 2|A_{\tzero}| - |B_{\tzero}|} := w_E$.\label{i2}
\item $U^t f(\min C_0) - U^t f(1) \geq \frac{2 |E_0|}{n + |E_0|}$.\label{i3}
\end{enumerate}

We will show that
\begin{align*}
\begin{rcases}
\text{I}_t \\
\mathcal{N}_{\max A_0}(U^t f) \cap E_0 = \emptyset
\end{rcases}
\implies
\begin{cases}
\text{I}_{t+1} \\
\text{II}_{t+1} \\
\text{III}_{t+1} \text{ or } L_{t+1} < 2.
\end{cases}
\end{align*}
Intuitively, for as long as $A_0$ and $E_0$ remain out of sight, the sets $C_t$ continue to grow and the spans of opinions in $A_0$ and $E_0$ don't exceed some fixed small amount. Note that $\text{II}_t$ is \emph{not} assumed to hold at $t=0$, rather the structure of the profile causes the agents in $A_0$ and $E_0$ to already become more clustered at the first time step.

First, we show that the two assumptions imply $\text{II}_{t+1}$.

Since, per definition, all agents in $A_0$ can see all agents in $C_t$ and we assume $E_0$ is out of sight of $A_0$, the case that pulls the two ends of $A_0$ as far apart as possible is that where the set difference of their respective neighbour sets is all of $D_t$, and where all agents in $D_t$ are placed at $U^tf(\max A_0) + 1$. 
 We obtain the bound
\begin{align*}
\diam{t+1}{A_0} & \leq \frac{|D_t|(1 + \diam{t}{A_0})}{|A_0|+|B_0|+|C_0|+|D_0|} = \frac{|D_t|(1 + \diam{t}{A_0})}{(n - |E_0|) / 2}.
\end{align*}
Since, by $\text{I}_{t}$ and \eqref{eq:6tocons1}, $C_t$ is nonempty, we have the trivial bound $\diam{t}{A_0} \leq 1$. Hence
\begin{align*}
\diam{t+1}{A_0} & \leq \frac{2 |D_t|}{(n - |E_0|) / 2} \stackrel{\text{I}_{t}}{ \leq}  \frac{4 |D_{\tzero}|}{n - |E_{\tzero}|},
\end{align*}
as desired.

Similarly we see that, by definition, $\mathcal{N}_{\min E_0} (U^t f) \setminus \mathcal{N}_{\max E_0} (U^t f) \subseteq B_t$. Hence, in the next time step, it is impossible to separate the two ends of $E_0$ more than if all of $B_t$ is visible to the agent $\min E_0$, while none of $B_t$ is visible to the agent $\max E_0$. Further, since $f$ is symmetric and $E_0$ is symmetric around the midpoint, all of $E_0$ can see a mirror image of $C_t$ and $D_t$ to the right. By the same argument as before, $\diam{t}{E_0} \leq 1$. 
We obtain the bound
\begin{align*}
\diam{t+1}{E_0} & \leq 2 \frac{|B_t| (1 + \diam{t}{E_0})}{|B_t|+2|C_t|+2|D_t|+|E_t|} \leq \frac{4 |B_t|}{n - 2|A_0| - |B_t|} \\
& \stackrel{\text{I}_{t}}{\leq} \frac{4 |B_{\tzero}|}{n - 2|A_{\tzero}| - |B_{\tzero}|},
\end{align*}
as desired. This proves $\text{II}_{t+1}$.

Next we deduce $\text{I}_{t+1}$. We will see that $C_{t+1} \supseteq C_t$ by checking that every agent in $A_0 \cup E_0$ remains in sight of every agent in $C_t$ after the update.

Since $A_0$ and $E_0$ remain out of sight at time $t$, an agent in $C_t$ cannot see the mirror image of $A_0$, but could a priori see every other agent. Since the profile is symmetric
\begin{align*}
U^{t+1} f(\max C_t) \leq \frac{(\frac{L_t}{2} + U^t f(1))(n - 2|A_0|) + |A_0| \langle U^t f(i) \rangle_{i \in A_0}}{n-|A_0|}.
\end{align*}
Since, as already noted, $\diam{t}{A_0} \leq 1$, it is clear that $U^t f(1) \leq \langle U^t f(i) \rangle_{i \in A_0} \leq U^{t+1} f(1)$. Hence,
\begin{align}
U^{t+1} f(\max C_t) - U^{t+1} f(1) \leq \frac{\frac{L_t}{2}(n- 2|A_0|)}{n-|A_0|} \leq \frac{L_0}{2}\frac{n - 2|A_0|}{n - |A_0|} \leq 1,
\end{align}
where the last inequality holds by \eqref{eq:6tocons4}. This shows that every agent in $A_0$ remains in sight of every agent in $C_t$ after the update.

We proceed similarly to prove retained contact with $E_0$. To simplify the notation, we assume without loss of generality that the profile is shifted to be symmetric around $0$, so that $f(1)+f(n) = 0$. A lower bound on the opinion of agent $\min C_t$ at time $t+1$ is obtained by placing all agents in $A_0 \cup B_t \cup C_t \cup D_t = A_0 \cup B_0 \cup C_0 \cup D_0$ at $-\frac{L_t}{2}$ and assuming this agent cannot see beyond $E_0$. This gives
\begin{align*}
U^{t+1} f(\min C_t) \geq
-\frac{|A_0| \cup |B_0| \cup |C_0| \cup |D_0|}{(n + |E_0|)/2} \frac{L_t}{2} \geq
-\frac{n-|E_0|}{n + |E_0|} \frac{L_0}{2}.
\end{align*}

Hence,
\begin{align*}
U^{t+1} f(\max E_0) - U^{t+1} f(\min C_t) &
\leq \frac{\diam{t+1}{E_0}}{2} + \frac{n-|E_0|}{n + |E_0|} \frac{L_{\tzero}}{2} \leq 1,
\end{align*}
where the last inequality holds by $\text{II}_{t+1}$ and \eqref{eq:6tocons5}. This establishes $\text{I}_{t+1}$.

Finally, we deduce that either $L_{t+1} < 2$ or $\text{III}_{t+1}$ holds. By definition of $C_t$,
\begin{align}\label{eq:extremocentric}
U^{t+1} f(1) = \langle U^t f(i) \rangle_{i \in A_0 \cup B_t \cup C_t}.
\end{align}
On the other hand, as the agent $\min C_t$ can see at least all of $E_0$,
\begin{align*}
U^{t+1} f(\min C_t)
& \geq
\frac{(|A_0| \!\! + \!\! |B_t| \!\! + \!\! |C_t|) \langle U^t f(i) \rangle_{i \in A_0 \cup B_t \cup C_t} \!\! + \!\! |D_t| \langle U^t f(i) \rangle_{i \in D_t} \!\! + \!\! |E_0| \langle U^t f(i) \rangle_{i \in E_0}}{|A_0 \cup B_0 \cup C_0 \cup D_0 \cup E_0|}.\\
\end{align*}
Hence,
\begin{align}
U^{t+1} f(\min C_t) - U^{t+1} f(1)
& \geq  \frac{2 |E_0|}{n + |E_0|} \left( \langle U^t f(i) \rangle_{i \in E_0} - \langle U^t f(i) \rangle_{i \in A_0 \cup B_t \cup C_t} \right)\\
& \stackrel{\eqref{eq:extremocentric}}{=}
 \frac{2 |E_0|}{n + |E_0|} \frac{L_{t+1}}{2}.
\end{align}
This proves that either $L_{t+1} < 2$ or $\text{III}_{t+1}$ holds.

~

{\bf Step 2:}
We show that
\begin{align*}
\begin{rcases}
\text{I}_{t} \\
\text{III}_{t} 
\end{rcases}
\implies
U^{t+1} f(1) - U^{t} f(1) \geq s \geq w_A + \frac{w_E}{2}
\end{align*}
where $s>0$ is bounded away from $0$, independent of $t$, by some function which only depends on the relative sizes $\frac{|C_0|}{n}$ and $\frac{|E_0|}{n}$.

By definition, at time $t$, the left extremist can see all agents in $A_0 \cup B_t \cup B_t$. Hence,
\begin{align}
U^{t+1} f(1) & = \langle U^t f(i) \rangle_{i \in A_0 \cup B_t \cup C_t}\\
& \stackrel{\text{I}_{t}}{\geq} \langle U^t f(i) \rangle_{i \in A_0 \cup B_0 \cup C_0}\\
& \stackrel{\text{III}_{t}}{\geq} \frac{U^t f(1) |A_0 \cup B_0| + \left( U^t f(1) + \frac{2 |E_0|}{n+|E_0|} \right) |C_0|}{|A_0 \cup B_0 \cup C_0|}\\
& = U^t f(1) + \frac{2 |C_0| |E_0|}{(|A_0| +| B_0| +| C_0|)(n + |E_0|)} 
\end{align}

\begin{align}
\implies U^{t+1} f(1) - U^{t} f(1) \geq \frac{2 |C_0| |E_0|}{(|A_0| +| B_0| +| C_0|)(n + |E_0|)} := s.
\end{align}

By \eqref{eq:6tocons1} and the fact that $|A_0| + |B_0| + |C_0| \leq \frac{n - |E_0|}{2}$, $s$ is bounded away from 0, independent of $t$, by some function which only depends on the relative sizes $\frac{|C_0|}{n}$ and $\frac{|E_0|}{n}$.

Finally, by \eqref{eq:6tocons6} and the definitions of $w_A$ and $w_E$ in $\text{II}_t$, we have $s \geq w_A + \frac{w_E}{2}$ as desired.

~

{\bf Step 3:}
By Steps 1 and 2 there is some $t^*$, depending only on the relative sizes $\frac{|C_0|}{n}$ and $\frac{|E_0|}{n}$, such that $A_0$ and $E_0$ can first see each other at time $t^*$.

We claim that $L_{t^*+1} \leq 2$. If $L_{t^*} \leq 2$ this is trivial, as the diameter $L_t$ is a non-increasing function of time. If $L_{t^*} > 2$ then, since $A_0$ and $E_0$ can see each other, we must have
\begin{align}
L_{t^*} \leq 2 + 2 \left( \diam{t} {A_0} + \frac{\diam{t} {E_0}}{2} \right).
\end{align}
Since $\text{II}_{t^*}$ holds, Step 2 will imply that $L_{t^*+1} \leq 2$.

Finally, since $E_0$ is symmetric about the midpoint and visible to all agents, consensus must follow in a time that is still bounded by some function of the relative sizes $\frac{|A_0|}{n}$, $\frac{|C_0|}{n}$ and $\frac{|E_0|}{n}$, by an argument similar to that in the proof of Corollary \ref{cor:regularity}.
\end{proof}
\end{taggedblock}

\begin{corollary}\label{cor:6tocons}
Let $k\geq 0$ and $n \geq 3$ where $n$ is odd. Let $f$ be a symmetric profile on $n$ agents and let $g$ be a profile on $n + (n-1)k$ agents such that $e^0 := B_k^n(g)-f$ is a consistent deviation.

Assume there exists a time $t_0$ such that the following hold:

\newcommand{\newsetsplus}[1]{  #1^+}
\newcommand{\newsetsminus}[1]{ #1^-}

\begin{enumerate}
\item $L_{{t_0}} + e_l^{t_0}(1) + e_r^{t_0}(n) \leq 4$, where $e_l^{t_0}$ and $e_r^{t_0}$ are computed recursively acording to \eqref{eq:rightbound} and \eqref{eq:leftbound} in Theorem \ref{cor:devprop}.
\item\label{it:6tocons2} There is, for the profile $U^{t_0} f$, a choice of subsets $A_{t_0}$, $B_{t_0}$, $C_{t_0}$, $D_{t_0}$, $E_{t_0}$ of $[n]$ such that
\begin{itemize}
\item parts (i)-(iv) of Definition \ref{def:good} are satisfied,
\item $\min \{|A_{t_0}|,\, |C_{t_0}|,\, |E_{t_0}| \} \geq 2$,
\item the following inequalities hold, where we denote $\newsetsminus{X} := |X_{t_0}| - 1$, $\newsetsplus{X} := |X_{t_0}| + 1$ and $X^0 := |X_{t_0}|$ for $X = A, B, C, D, E$:
\begin{align}
\frac{n-2 \newsetsminus{A}}{n- \newsetsminus{A}} \frac{L_{t_0}}{2} &
\leq 1 \\
\frac{n- \newsetsminus{E}}{n +  \newsetsminus{E}} \frac{L_{t_0}}{2} + \frac{2  \newsetsplus{B}}{n - 2 A^0 -  \newsetsplus{B}} &
\leq 1\\
\frac{2  \newsetsplus{B}}{n - 2 A^0 -  \newsetsplus{B}} + \frac{4  \newsetsplus{D}}{n -  \newsetsminus{E}} &
\leq  \frac{2  \newsetsminus{C} \newsetsminus{E}}{( \newsetsminus{A} + B^0 + \newsetsminus{C})(n + \newsetsminus{E})}.
\end{align}
\end{itemize}
\item $U^{t_0} f(\max C_{t_0}) + e_r^{t_0}(\max C_{t_0})       \leq       U^{t_0} f(1) - e_l^{t_0}(1) + 1 \label{cor:6tocons1}$.
\item $U^{t_0} f(\min C_{t_0}) - e_l^{t_0}(\min C_{t_0})       \geq       U^{t_0} f(\max E_{t_0}) + e_r^{t_0}(\max E_{t_0}) - 1$.
\item $U^{t_0} f(\max A_{t_0}) + e_r^{t_0}(\max A_{t_0})       <          U^{t_0} f(\min E_{t_0}) - e_l^{t_0}(\min E_{t_0}) - 1\label{cor:6tocons3}$.
\end{enumerate}

Then there is some $T$, depending only on $t_0$ and the relative sizes  $\frac{|A_{t_0}|}{n}$, $\frac{|C_{t_0}|}{n}$ and $\frac{|E_{t_0}|}{n}$, such that $U^T g$ is a consensus. In particular, $T$ is otherwise independent of $k$ and $g$.
\end{corollary}
\begin{proof}
Suppose the conditions (i)-(v) hold for $f$. We claim that the hypotheses of Theorem \ref{thm:6tocons} hold for $U^{t_0} g$, with $A_0$ and $E_0$ chosen such that their respective coarsenings are $A_{t_0}$ and $E_{t_0}$. For if $C_0$ is defined as in Definition \ref{def:good}(v), then conditions (iii)-(v) guarantee that its coarsening contains $C_{t_0}$. By condition (ii), this is enough, since \eqref{eq:6tocons1}-\eqref{eq:6tocons6} continue to hold if the size of $C_0$ is increased at the expense of $B_0$ and $D_0$. Note that, since refining leads to changes in the relative sizes of the sets, the condition (ii) is what is needed for \eqref{eq:6tocons4}-\eqref{eq:6tocons6} to hold.
\end{proof}

\begin{corollary}
There is a time $T$ such that $U^T f^6$ is a consensus.
\end{corollary}
\begin{proof}

We apply Corollary \ref{cor:6tocons} with
\begin{align}
n &= 80~005\\
f &= f^{n,6}\\
g &= (f^{n,6})^{(k)}\\
e^0_l &= e^0_r \equiv 0 \\
t_0 &= 8.
\end{align}

The computation (see the ancillary file \texttt{L6.jl} for the code) shows that the conditions of Corollary \ref{cor:6tocons} are then satisfied for good sets with sizes
\begin{align}
|A_8| =  31~537 ~~ |B_8| = 0 ~~ |C_8| = 40  ~~ |D_8| = 3 ~~ |E_8| = 16~845.
\end{align}
\end{proof}

\section{Final remarks}
The two mechanisms behind achieving consensus discussed in Sections 4 and 5 are quite different, and the case $L=6$ is clearly the harder one.

One aspect of this is that consensus at $L=6$ seems to be much harder to detect if one simulates random profiles instead of equally spaced ones. In our own experiments, we only detect the consensus strikes back phenomenon for random profiles when the number of agents simulated was in the milions. A heuristic explanation for this is that the structure with microclusters is more sensitive to the noise that comes if the initial opinions are chosen randomly. Even if the initial profile is quite uniform at $t=0$, the HK-updating operator tends to amplify what little unevenness there is, so that even if some update of the profile has approximately the right structure, the sizes of the ``clusters'' meant to correspond to $A$, $B$, $C$, $D$ and $E$ might vary more than one might first guess. In particular, the absence of guaranteed symmetry in a random profile is problematic. There are then two things that can go wrong on the way to consensus. The first is that one microcluster ends up significantly larger than the other, in which case the extremists on one side will be pulled faster towards the centre and the centre in turn will be drawn to this side. The second is that, even if the microclusters end up about the same size, fragmentation might follow if the sizes of the extremist clusters are too different. The smaller one will be more easily affected by its microcluster, which will also place itself further away, and might thus win the race to the centre and steal away the central cluster.

The only chance for a uniformly random profile with $L=6$ to reach consensus is to have the two extremist clusters see the centre at \emph{exactly} the same time, and fragmentation of uniformly random profiles might happen even for $n$ in the millions.

~

Two possible lines of further research are
\begin{itemize}
\item To understand what happens when $L \in [L_1,\,L_2]$.
\item To apply our method of comparing a profile to its refinements to distributions other than the uniform. 
\end{itemize}

\section{Acknowledgements}
I would like to thank my supervisor Peter Hegarty for reviewing several drafts of this paper, for helpful discussions and advice. I would also like to thank Martin Raum and Tommy Vågbratt for advice on the code.

\newpage

\begin{alphasection}
  \setcounter{section}{1}
  \section*{Appendix A: Explicit formulas for the first two updates of an equally spaced profile } 
  \begin{definition}
For an equally spaced profile $f$, the difference in opinion between two consecutive agents will be referred to as the \emph{separation parameter}, or just \emph{separation}, of $f$, and will be denoted by $d$. When used about updates of equally spaced profiles, it refers to the parameter of the original profile.
\end{definition}

\begin{prop}\label{lem:step1}
Let $f = f^{n,L}$ such that $\frac{1}{d}$ is an integer and $L \geq 2$. 
Then the symmetric profile $U f$ is given, for $i \leq \frac{n}{2}$, by
\begin{align*}
U f(i) =& \frac{d(i-1) + 1}{2} & \text{if } & i \leq \frac{1}{d} + 1 \\
U f(i) =& f(i) = d(i-1) & \text{if } & i \geq \frac{1}{d} + 1.
\end{align*}
Note in particular that, for any $k \geq 0$, $B_k^n(Uf^{(k)}) = Uf$.
\end{prop}
\begin{taggedblock}{proofsuniform}
\begin{proof}
First assume $i \leq \frac{1}{d} + 1$. Then, the leftmost neighbour of $i$ has opinion 0 and, since $\frac{1}{d}$ is an integer and $L \leq 2$, its rightmost neighbour has opinion $f(i) + 1$. Let $m = i + \frac{1}{d}$ denote its total number of neighbours. Then,

\begin{align*}
Uf(i) = &
\frac{1}{m} \sum_{j=0}^{m-1} jd =
d \left(\frac{1}{m} \sum_{j=0}^{m-1} j \right) =
d \left(\frac{1}{m} \frac{(m-1)m}{2} \right)\\
 & = \frac{d(m - 1)}{2} =
\frac{d(i-1) + 1}{2},
\end{align*}   
as stated.

If $\frac{1}{d} + 1 \leq i \leq \frac{n}{2}$ then, since $L \geq 2$, $i$ has the same configuration of neighbours on both sides, and hence its opinion won't change.
\end{proof}
\end{taggedblock}

\begin{prop}\label{lem:step2}
Let $f = f^{n,L}$ with $L \geq 4$ and such that $\frac{1}{d}=\frac{n-1}{L}$ is an even integer.

Then the symmetric profile $U^2f$ is given, for $i \leq \frac{n}{2}$, by
\begin{align}
U^2f(i)= & \frac{d \underline{i}^2 + 6\underline{i} + \frac{11}{d} - d + 2}{4 \left(\underline{i} + \frac{3}{d} + 1 \right)} & \text{if } & i \leq \frac{1}{d} + 1 & \label{eq:t2first}\\
U^2f(i)= & \frac{di^2 +2i - di + \frac{3}{2d} -\frac{1}{2}}{\frac{2}{d} + 2i} & \text{\!\!\!\!\!\!if } \frac{1}{d} + 1 \leq & i \leq \frac{3}{2d} + 1 &\label{eq:t2second}\\
U^2f(i)= & \frac{-di^2 + 6i + 4di - 3d - 6}{2\left( \frac{4}{d} + 2 -i \right)} & \text{\!\!\!if } \frac{3}{2d} + 1 \leq & i \leq \frac{2}{d} + 1 & \label{eq:t2third}\\
U^2f(i)= & (i-1)d & \text{\!\!\!if } \frac{2}{d} + 1 \leq & i \leq \frac{n}{2}, & \label{eq:t2fourth}
\end{align}
where $\underline{i} = 2 \left\lfloor \frac{i-1}{2} \right\rfloor + 1$ is used to denote rounding down to the nearest odd integer.
\end{prop}

\begin{taggedblock}{proofsuniform}
\begin{proof}
First note that, since $L \geq 4$, \eqref{eq:t2fourth} follows immediately from Proposition \ref{lem:step1} which says that, for any neighbour $j$ of any agent $i$ such that $\frac{2}{d} + 1 \leq i \leq \frac{n}{2}$, we have $Uf(j) = f(j)$.

The proofs of \eqref{eq:t2first}-\eqref{eq:t2third} are a straightforward if messy computation.

For each agent $i \leq \frac{2}{d} + 1$, we define sets 
\begin{enumerate}[label={\alph*)}]
\item $S_i^- = \{j \in \mathcal{N}_i(Uf) : f(j) \leq 1\}$
\item $S_i^+ = \{j \in \mathcal{N}_i(Uf) : f(j) \geq 1\}$.
\end{enumerate}
Since, by Proposition \ref{lem:step1}, $Uf (\frac{2}{d} + 1) = 2$, both sets are non-empty for each $i$. Denote
\begin{align*}
m_i^- &= |S_i^-|, ~~~~  m_i^+ = |S_i^+| \\
a_i^- &= \langle Uf(j)  \rangle_{j \in S_i^-}, ~~~~    a_i^+
= \langle Uf(j) \rangle_{j \in S_i^+}.
\end{align*}

Then
\begin{equation}\label{eq:viktatmedel}
U^2f(i) = \frac{m^-_i a^-_i + m^+_i a^+_i -1}{m^-_i + m^+_i - 1},
\end{equation}
where the subtracted ones serve to compensate for the fact that we count the agent with opinion $1$ twice.

By Proposition \ref{lem:step1}, the profile $Uf$ is equally spaced above $1$, with separation $d$. Hence
\begin{align}
m^+_i &=  \left\lfloor \frac{Uf(i)}{d} \right\rfloor + 1,\\
a^+_i &= \frac{2+d \left\lfloor \frac{Uf(i)}{d} \right\rfloor}{2}.
\end{align}
If an agent $i$ can see the left extremist in $Uf$, Proposition \ref{lem:step1} implies that $m^-_i = \frac{1}{d}+1$ and $a^-_i = \frac{3}{4}$.

In completing the calculation of $U^2f(i)$ for $i \leq \frac{2}{d} + 1$, we can now distinguish three cases depending on $i$.

{\bf Case 1:} $i \leq \frac{1}{d} + 1$.

We have $Uf(i) = \frac{1}{2} + d\left( \frac{i-1}{2} \right)$, so
\begin{equation*}
m^+_i =
\left\lfloor \frac{1}{2d} + \frac{i-1}{2} \right\rfloor + 1 \stackrel{\frac{1}{d} \text{ even}}{=}
\frac{1}{2d} + \left\lfloor \frac{i-1}{2} \right\rfloor + 1 =
\frac{1}{2} \left( \frac{1}{d} + 1 + \underline{i} \right),
\end{equation*}
and $a^+_i = \frac{1}{4} \left(5 + (\underline{i}-1)d \right)$. Further, $m_i^- = \frac{1}{d} + 1$ and $a_i^- = \frac{3}{4}$.

Inserting these into \eqref{eq:viktatmedel} gives \eqref{eq:t2first}.

{\bf Case 2:} $\frac{1}{d} + 1 \leq i \leq \frac{3}{2d} + 1$.

 By Proposition \ref{lem:step1}, $i$ can still see the extremist, so $m_i^- = \frac{1}{d} + 1$ and $a_i^- = \frac{3}{4}$. Further, $Uf(i) = (i-1)d$, so $m^+_i = i$ and $a^+_i = \frac{1}{2} \left( 2 + id - d \right)$.

Inserting these into \eqref{eq:viktatmedel} gives \eqref{eq:t2second}.

{\bf Case 3:} $\frac{3}{2d} + 1 \leq i \leq \frac{2}{d} + 1$.

We have $m^+_i = i$ and $a^+_i = \frac{1}{2} \left( 2 + id - d \right)$ as above. Starting from $i = \frac{3}{2d}+1$, every step inwards removes two neighbours to the left and increases $a_i^-$ by $\frac{d}{2}$, so
\begin{align*}
m_i^- &= \frac{1}{d} + 1 - 2\left(i - \left(\frac{3}{2d} + 1\right) \right) \\
a_i^- &= \frac{3}{4} + \frac{d}{2} \left( i - \left( \frac{3}{2d} + 1 \right) \right) = \frac{}{} \frac{(i - 1)d}{2}.
\end{align*}
Inserting these into \eqref{eq:viktatmedel} gives \eqref{eq:t2third}.
\end{proof}
\end{taggedblock}

\begin{rem}\label{rem:23over6}
If we only care about calculating the opinions of the first $\frac{1}{d} + 1$ agents, the assumption of $L \geq 4$ in Proposition \ref{lem:step2} can be relaxed to $L \geq 3$ without affecting the proof of \eqref{eq:t2first}. Substituting $i = 1$ into \eqref{eq:t2first} shows that $U^2 f^L (1) \geq \frac{11}{12}$. Fix a rational $L \leq \frac{23}{6}$. There must exist an infinite increasing sequence $\left( n_i \right)_{i=1}^\infty$ such that all $f^{n_i,L}$ satisfy the conditions in the proposition, and using symmetry we get that $D(U^2 f^{n_i,L}) \leq \frac{23}{6} - 2 \frac{11}{12} = 2$ for all such $n_i$.

By Proposition \ref{lem:cont}, $L \leq \frac{23}{6} \implies D(U^2 f^L) \leq 2$. Hence, by Corollary \ref{cor:regularity} there exists a bounded $T$ such that $U^T f^L$ is a consensus for all $L \leq \frac{23}{6}$.
\end{rem}

\end{alphasection}


\begin{thebibliography}{ABRR} 

\bibitem{Julia} J. Bezanson, A, Edelman, S. Karpinski and V. B. Shah, \emph{Julia: A fresh approach to numerical computing}, SIAM Review \textbf{59} (2017), No. 1, 65--98.
  https://doi.org/10.1137/141000671

\bibitem{BBCN} A. Bhattacharya, M. Braverman, B. Chazelle and H. L. Nguyen, \emph{On the convergence of the Hegselmann-Krause system}, Proceedings of the 4th Innovations in Theoretical Computer Science Conference (ICTS 2013), Berkeley CA, January 2013.


\bibitem{BHT2} V. D. Blondel, J. M. Hendrickx and J. N. Tsitsiklis, \emph{On Krause's multi-agent consensus model with state-dependent connectivity}, IEEE Trans. Automat. Control \textbf{54} (2009), No. 11, 2586--2597.

  

  
\bibitem{WH2} P. Hegarty and E. Wedin, \emph{The Hegselmann-Krause dynamics for equally spaced agents}, J. Difference Equ. Appl. \textbf{22} (2016), No. 11, 1621--1645.
  
\bibitem{HK}{ R. Hegselmann and U. Krause, \emph{Opinion dynamics and bounded confidence: models, analysis and simulations}, Journal of Artificial Societies and Social Simulation \textbf{5} (2002), No. 3.} 


\bibitem{Arb} F. Johansson, \emph{Arb: efficient arbitrary-precision midpoint-radius interval arithmetic}, IEEE Transactions on Computers \textbf{66} (2017), No. 8, 1281-1292. DOI: 10.1109/TC.2017.2690633.
  


  
\bibitem{L} J. Lorenz, \emph{Consensus strikes back in the Hegselmann-Krause model of continuous opinion dynamics under bounded confidence}, Journal of Artificial Societies and Social Simulation \textbf{9} (2009), No 1, 105--120. 

  

\bibitem{MT} S. Mohajer and B. Touri, \emph{On convergence rate of scalar Hegselmann-Krause dynamics}, American Control Conference 2013, Washington DC (2013), pp. 206--210.

\bibitem{vdV1} A.W. van der Vaart, \emph{Asymptotic Statistics}, Cambridge University Press (1998).
  
\bibitem{WH1} E. Wedin and P. Hegarty, \emph{A quadratic lower bound for the convergence rate in the one-dimensional Hegselmann-Krause bounded confidence dynamics}, Discrete Comput. Geom. \textbf{53} (2015), No. 2, 478--486.

\bibitem{WH3} E. Wedin and P. Hegarty, \emph{The Hegselmann-Krause dynamics for the continuous-agent model and a regular opinion function do not always lead to consensus}, IEEE Trans. Automat. Control \textbf{60} (2015), No. 9, 2416--2421.






\end{thebibliography}
\end{document}